\documentclass[a4paper,11pt,english]{smfart}
\usepackage[latin1]{inputenc}
\usepackage[francais,english]{babel}
\usepackage{amsfonts}
\usepackage{amsmath} 
\usepackage{latexsym}
\usepackage{array}
\usepackage{amssymb}
\usepackage{textcomp}
\usepackage{pifont}
\usepackage{enumerate}
\usepackage{smfthm}
\usepackage{graphicx}
\usepackage{color}
\usepackage[T1]{fontenc}
\usepackage{hyperref} 
\theoremstyle{plain}

\newtheorem*{acknowledgements}{Acknowledgements}
\newtheorem{assumption}{Assumption}

\newcommand{\ligne}{\vspace{1\baselineskip}}
\newcommand{\ph}{\phantomsection}

\newcommand{\R}{  \mathbb{R}   }

\newcommand{\X}{  \mathcal{X}  }
\newcommand{\eps}{\varepsilon}
\newcommand{\e}{  \text{e}   }
\newcommand{\wt}{  \widetilde   }

\newcommand{\N}{  \mathbb{N}   }

\newcommand{\un}{\underline{\mathcal{N}}}

\newcommand{\D}{  \mathcal{D}   }

\newcommand{\dis}{\displaystyle}
\newcommand{\h}{  h }

\newcommand{\om}{  \omega   }

\renewcommand{\a}{  \alpha   }
\renewcommand{\b}{  \beta   }
\newcommand{\s}{  \sigma   }

\renewcommand{\phi}{  \varphi  }
\renewcommand{\L}{  \mathcal{L}   }

\newcommand{\<}{  \langle   }

\renewcommand{\>}{  \rangle   }
\renewcommand{\l}{  \ell  }
\renewcommand{\S}{  \mathbb{S}  }
\numberwithin{equation}{section}

 \author{ Beno\^it Gr\'ebert, Tiphaine J\'ez\'equel and Laurent Thomann}
\address{Laboratoire de Math\'ematiques J. Leray, Universit\'e de Nantes, UMR CNRS 6629\\
2, rue de la Houssini\`ere \\
44322 Nantes Cedex 03, France.}
\email{benoit.grebert@univ-nantes.fr,tiphaine.jezequel@univ-nantes.fr,laurent.thomann@univ-nantes.fr}

\title[Dynamics of Klein-Gordon]
{Dynamics of Klein-Gordon on a compact surface near a homoclinic orbit}

\begin{document}
\frontmatter

\begin{abstract}
We consider the   Klein-Gordon equation  (KG) on a Riemannian surface $M$
$$  \partial^{2}_t u-\Delta u-m^{2}u+u^{2p+1} =0,\quad p\in \N^{*},\quad  (t,x)\in \R\times M,$$
which  is globally well-posed in the energy space. This equation has a homoclinic orbit to the origin, and in this paper we study the dynamics close to it.  Using a strategy from Groves-Schneider,  we get the existence of a large family of heteroclinic connections to the center manifold that are close to the homoclinic orbit during all times. We point out that the solutions we construct are not small.
    \end{abstract}

\keywords{Klein-Gordon equation, wave equation, homoclinic orbit, center manifold.}
\altkeywords{  Equation de Klein-Gordon, \'equation des ondes, orbite homocline, vari\'et\'e centrale.}\frontmatter
\subjclass{ 37K45, 35Q55, 35Bxx}
\thanks{
\noindent B.G. was supported in part by the  grant ANR-10-BLAN-DynPDE.\\
 B.G. and L.T.  were supported in part by the  grant ANR-10-JCJC 0109.}

\maketitle

\tableofcontents

\section{Introduction, statement of the main results}
\subsection{General introduction} Denote by  $M$   a compact Riemannian manifold without boundary of dimension 1,2 or 3 and denote by $\Delta=\Delta_{M}$ the Laplace-Beltrami operator on $M$. In this paper we are concerned with the following nonlinear Klein-Gordon  (KG) equation
\begin{equation}\label{kg} \tag{KG}
\left\{
\begin{aligned}
& \partial^{2}_t u-\Delta u-m^{2}u+u^{2p+1} =0,\quad 
(t,x)\in\R\times M,\\
&u(0,x)= u_{0}(x),\quad \partial_{t}u(0,x)= u_{1}(x),
\end{aligned}
\right.
\end{equation} 
where $p\geq 1$ is an integer, and $(u_{0},u_{1})\in H^{1}(M)\times L^{2}(M)$ are real-valued.\ligne

 It is well-known that there exists a Hilbert basis of $L^{2}(M)$ composed with eigenfunctions $(e_{n})_{n\geq 0}$ of~$\Delta$. Moreover (see {\it e.g.} \cite{Helffer}), there exists a sequence $0=\lambda_{0}<\lambda_{1}\leq \dots \leq \lambda_{n}\leq \dots$ so that  
 \begin{equation*}
 -\Delta e_{n}=\lambda^{2}_{n}e_{n}, \quad n\geq 0.
 \end{equation*}
In the sequel, we define the scalar product on $L^{2}(M)$   by $\dis\<f,g\>=\frac1{{\rm Vol}\, M}\int_{M}fg$, where ${\rm Vol}\, M$ denotes  the volume of $M$, we assume that $\|e_{n}\|_{L^{2}}=1$ and we set $e_{0}=1$.   \ligne

We make following  assumptions 
\begin{assumption}\ph\label{assumption0}
The parameter  $m$ satisfies $\dis 0<m<\lambda_{1}.$
\end{assumption}

\begin{assumption}\ph\label{assumption}
The manifold $M$ and the integer $p$ satisfy either:\\[5pt]
\indent $\bullet$ $M$ is any compact manifold without boundary of dimension 1 or 2 and $p\geq 1$\\[5pt]
\indent $\bullet$ $M$ is any compact manifold without boundary of dimension 3 and $p= 1$.\\[5pt]
Moreover, up to a rescaling, we can assume that ${\rm Vol}\, M=1$.
\end{assumption}
~

The stationary solutions of (KG)  (solutions which only depend on the space variable)  are exactly the constants $u=0$, $u=m^{1/p}$ and $u=-m^{1/p}$ (see Lemma \ref{lem.equi}). The origin is an equilibrium with an unstable direction. In fact, the eigenvalues of $-\Delta-m^{2}$ are the $(\lambda_{k}^{2}-m^{2})_{k\in \N}$. Since~$0<m<\lambda_{1}$,  the case $k=0$ only,  gives the hyperbolic directions, corresponding to the solution $\exp(mt)$ for $t>0$ (resp. $\exp(-mt)$ for $t<0$). It turns out that \eqref{kg} admits a homoclinic orbit  to the origin. Indeed, the following $x-$independent function is a solution to \eqref{kg}
\begin{equation*}
\a(t)=\frac{m^{1/p}(p+1)^{1/(2p)}}{\big(\cosh(pmt)\big)^{1/p}},
\end{equation*}
and in the sequel we will refer to 
$$h(t)=(\a(t),\beta(t)),\quad \text{with}\quad \dot{\beta}(t)=\a(t),$$
as the time homoclinic solution (to the origin) of \eqref{kg}, see Section \ref{Sect.14}.  In this work we aim to study the dynamics of \eqref{kg} near this particular trajectory, and we will show the existence of solutions which remain close for all times to this exact temporal solution. 
\ligne

In the case $M=\S^{1}$ we can precise the dynamics around the equilibrium $u=m^{1/p}$ (the study near $u=-m^{1/p}$ is similar since the non-linearity is odd). We linearize the equation near this point ($u=m^{1/p}+v$) and we are lead to study the spectrum of $(-\partial^{2}_{x}+2pm^{2})v$. This   operator  on $\S^1$   is self adjoint and has pure point spectrum $j^2+2pm,\ j\in\N$. Using the Birkhoff normal form theory,    Bambusi~\cite{Bam04} (see also \cite{BG06}) has shown  that, for a generic choice of $m$ (in order to avoid resonances between the frequencies $\sqrt{j^2+2pm},\ j\in \N$), the solutions of \eqref{kg} with initial datum of the form $m^{1/p}+v_0$ with~$v_0$ small enough,  remain close to the equilibrium point $m^{1/p}$ for very long time. See also the work \cite{Delort09} of J.-M. Delort for quasi-linear equations.

Moreover, using  KAM theory, C. E.  Wayne \cite{Wayne} and J. P\"oschel \cite{Posch} have proved that, for a generic choice of $m$, there exist many quasi-periodic solutions near $w=m^{1/p}$. See also the recent work of Berti-Biasco-Procesi \cite{BBP1} for derivative wave equations. Observe that quasi-periodic solutions can be constructed even if the equilibrium has a finite number of hyperbolic directions (see \cite{BertiBolle}).

In higher dimension and for a general manifold, few is known. In the case $M$ is a Zoll manifold, Bambusi-Delort-Gr\'ebert-Sjeftel \cite{BDGS} have developed a Birkhoff normal form theory for \eqref{kg} near an elliptic equilibrium. Up to now, there is no KAM-type result for \eqref{kg} in dimension greater than two.

In this work, we describe some possible behaviours of \eqref{kg} near the homoclinic orbit. We state the existence of solutions that travel from a neighbourhood of the origin to turn around the equilibrium~$m^{1/p}$, close to the homoclinic connection to 0 (see Theorems \ref{thm1} and \ref{coro.thm}). We stress out that these results do not require non resonance conditions. 

The existence of homoclinic or heteroclinic connections to periodic or quasi-periodic solutions is a question of interest: for example, in the case of parabolic PDE's on a one dimensional bounded domain, homoclinic orbits to equilibrium points can exist but there cannot exist any homoclinic connections to periodic orbits (see \cite{JolyRaugel}). In some other cases some homoclinic connections to small periodic orbits exist while homoclinic orbits to 0 do not exist or remain an open question (see for instance \cite{LombardiArticle,LombardiLivre,TheseTiphaine}). More generally, such homoclinic connections to the center manifold often appear when a homoclinic orbit does not persist after a perturbation of the system (see for instance \cite{ShaZeng}).  We refer also to \cite{BertiCarminati}, where the existence of homoclinic  solutions to a beam equation is studied. 

In this paper, we get the existence of a large family of homoclinic connections to the center manifold (more precisely heteroclinic connections to some solutions lying in the center manifold). The proof is based on a perturbative method which is classical (see for instance \cite{IoossPeroueme,LombardiArticle}) for finite dimensional reversible systems (systems which anticommute with a symmetry $S$): the key idea is that if a solution hits the reversibility plane $\{u \mid
u=Su\}$, then the latter solution is reversible. This method requires more computations in the infinite dimensional cases. It was already performed in a situation  close to ours by M. Groves and G. Schneider \cite{GS1,GS2,GS3}: but in their case, they work in the neighbourhood of a bifurcation and get small homoclinic solutions, while in our case the size of the solutions is of order~1 and this requires some additional work on the linearised system, which is one of our main contributions in this paper.

Finally we also mention the recent book of Nakanishi \& Schlag \cite{NS} on invariant manifolds in the context of dispersive Hamiltonian PDEs.

\subsection{A first motivation for studying $\eqref{kg}$}   
Recall that the usual non linear Klein-Gordon equation reads
\begin{equation}\label{NLW.0}
\partial^{2}_{t}w-\Delta w+w-f(w)=0,\quad (t,x)\in \R\times M,
\end{equation}
where $f$ is a non linear function. Let us show that if there exists a nonzero equilibrium, the equation near the smallest equilibrium is of the form \eqref{kg} (but with a general non linearity). 
The equilibrium of \eqref{NLW.0} are the real constants $w_0$ satisfying $f(w_0)-w_0=0$. Observe that if one performs the change of coordinates $u=w-w_0$, then there exists a non linear function $g$ such that $u$ satisfies 
\begin{equation*}
\partial^{2}_{t}u-\Delta u-(f'(w_0)-1)u-g(u)=0.
\end{equation*}
Given that $f$ is non linear, $f'(0)=0$ and $f'(0)-1<0$. Thus, if $f$ is $\mathcal{C}^1$ and if there exists an equilibrium $w_0\neq 0$, then by the intermediate value theorem, we get that the smallest equilibrium $w_0$ satisfies $f'(w_0)-1\geq 0$. This means that near the smallest non zero equilibrium of \eqref{NLW.0}, the equation is of the form
$$\partial^{2}_t u-\Delta u-m^{2}u-g(u)=0.$$
For instance in this paper we consider a nonlinearity $g(u)=u^{2p+1}$, $u=m^{1/p}$ is a (stable) equilibrium and the change of variable $u=m^{1/p}+w$ leads to
$$\partial^{2}_{t}w-\Delta w+2m^2 w-f(w)=0,\quad (t,x)\in \R\times M$$
with a nonlinearity $\dis f(w)=(w+m^{\frac{1}{p}})^{2p+1}-m^{\frac{2p+1}{p}}-(2p+1)m^2w$. 
\subsection{Another motivation of the problem} Here we are inspired from the works \cite{GS1,GS2,GS3} of M. Groves and G. Schneider.
Consider the non linear wave equation on the real line 
\begin{equation}\label{NLW}
\partial^{2}_{t}w-\partial^{2}_{x}w+w-w^{2p+1}=0,\quad (t,x)\in \R\times \R.
\end{equation}
One can be interested in the construction of  ``modulated pulse solutions '' for \eqref{NLW} which are solutions of the form \begin{equation}\label{lorentz}
w(t,x)=u\big(\frac{x-\beta t}{c_{\beta}},\frac{t-\beta x}{c_{\beta}}\big),
\end{equation}
where $u(s,y)$ is a $2\pi-$periodic function in $y$, and where $\beta \in (0,1)$ and $c_{\b}>0$. In the particular case  $c_{\b}=(1-\b^{2})^{1/2}$, \eqref{lorentz} is the Lorentz transform which preserves \eqref{NLW} and in general we get
\begin{equation}\label{kg*}
\partial^{2}_{s}u-\partial^{2}_{y}u-\frac{c^{2}_{\b}}{1-\b^{2}}\big(u-u^{2p+1}\big)=0 ,\quad (s,y)\in \R\times \S^{1}.
\end{equation}
Now let $m\in (0,1)$ and $\beta \in (0,1)$, then we can choose $c_{\b}>0$ so that $\frac{c^{2}_{\b}}{1-\b^{2}}=m^{2}\in (0,1)$, which is~\eqref{kg} up to the change of unknown $u\longmapsto u^{1/p}w$. 
  As a consequence, for each $m\in (0,1)$, any solution to \eqref{kg*} provides a one-parameter family of solutions to \eqref{NLW}.

\subsection{Hamiltonian structure of \eqref{kg}} 
Denote by $v=\partial_{t}u$ and introduce
\begin{equation}\label{Hami0}
H=\frac12\int_{M}\Big(|\nabla_{x}u|^{2}+v^{2}-m^{2}u^{2}\Big)+\frac1{2p+2}\int_{M}u^{2p+2}.
\end{equation}
Then, the system \eqref{kg} is equivalent to 
\begin{equation}\label{ham}
\dot{u}=\frac{\delta H}{\delta v},\quad \dot{v}=-\frac{\delta H}{\delta u}.
\end{equation}
We write 
\begin{equation*}
u(t,x)=\sum_{n=0}^{\infty}a_{n}(t)e_{n}(x),\quad  v(t,x)=\sum_{n=0}^{\infty}b_{n}(t)e_{n}(x),
\end{equation*}
where 
\begin{align*}
(a_n)_{n\in\N}\in  h^1(\N,\R)&:=\big\{ x=(x_n)_{n\in\N} \mid \|x\|^{2}_{h^1}=\sum_{n\in\N}(1+\lambda_{n}^2)|x_n|^2 <\infty \big\},\\
(b_n)_{n\in\N}\in  \l^2(\N,\R)&:=\big\{ x=(x_n)_{n\in\N} \mid \|x\|^{2}_{\l^2}=\sum_{n\in\N}|x_n|^2 <\infty \big\},
\end{align*}
in such a way that to the continuous phase space $\X:=H^{1}\times L^{2}$ corresponds the discrete one   $h^1\times \l^2$. We endow this space with the natural norm 
\begin{equation*}
\|X\|=\|u\|_{H^{1}}+\|v\|_{L^{2}},\quad \text{for}\quad  X=(u,v).
\end{equation*}
 We  define the two dimensional subspace $\X_{h}\subset \X$ spanned by the stable/unstable modes (hyperbolic modes) of $X$
\begin{equation}\label{defX0}
\X_{h}=\big\{X\in \X\; |\;\forall\,n\geq 1,\quad a_{n}=b_{n}=0 \big\},
\end{equation}
and we observe that $\X_{h}$ is an invariant subspace for \eqref{ham}. We also define  the infinite dimensional subspace $\X_{c}\subset \X$ spanned by the center modes (elliptic modes) of~$X$
\begin{equation}\label{defXc}
\X_{c}=\big\{X\in \X\; |\;a_{0}=b_{0}=0\big\}.
\end{equation}
We clearly have $\X=\X_{h} \oplus\X_{c}$. Denote by $P : \X\longrightarrow \X_{h}$ the projection onto $\X_{h}$ and define $Q=I-P :  \X\longrightarrow \X_{c}$ the projection onto $\X_{c}$. In the sequel, for $X\in \X$, we use the notation $X_{h}=PX$ and $X_{c}=QX$.

In the coordinates $(a_{n},b_{n})_{n\geq0}$, the Hamiltonian in \eqref{Hami0} reads 
\begin{equation}\label{Hami}
H=\frac12\sum_{n=0}^{\infty}\Big[(\lambda_{n}^{2}-m^{2})a^{2}_{n}+b^{2}_{n}\Big]+\frac1{2p+2}\int_{M}\Big(\sum_{k=0}^{\infty}a_{k}e_{k}(x)\Big)^{2p+2}\text{d}x,
\end{equation}
and the system \eqref{ham} becomes
 \begin{equation}\label{syst} 
\left\{
\begin{aligned}
&\dot{a}_{n} =b_{n},\quad n\geq 0\\
&\dot{b}_{n} =-(\lambda_{n}^{2}-m^{2})a_{n}-\int_{M}\Big(\sum_{k=0}^{\infty}a_{k}e_{k}(x)\Big)^{2p+1}e_{n}(x)\text{d}x,\quad n\geq 0.
\end{aligned}
\right.
\end{equation} 

\subsection{The homoclinic orbit}\label{Sect.14}
The space-stationary solutions of \eqref{kg} exactly correspond to the  solutions of \eqref{syst} satisfying $a_{n}=b_{n}=0$ for $n\geq 1$. In this case, the equation on $(a_{0},b_{0})$ reads 
 \begin{equation} \label{toy0}
\left\{
\begin{aligned}
&\dot{a}_{0} =b_{0}\\
&\dot{b}_{0} =m^{2}a_{0}- a^{2p+1}_{0}, 
\end{aligned}
\right.
\end{equation} 
and this system possesses a homoclinic solution to 0 which we will denote in the sequel by 
$$\h: t\mapsto (\alpha(t),\beta(t)),$$
 and we will denote by $\mathcal{K}_{0}$ the curve which is described (see Figure \ref{PhasePortraita0b0}). Indeed we can explicitly compute 
\begin{equation}\label{def.homo}
\a(t)=\frac{m^{1/p}(p+1)^{1/(2p)}}{\big(\cosh(pmt)\big)^{1/p}},\quad \b(t)=-m^{1/p+1}(p+1)^{1/(2p)}\frac{\sinh(pmt)}{\big(\cosh(pmt)\big)^{1/p+1}},
\end{equation}
and we have the bounds
\begin{equation}\label{ab}
|\a(t)|\leq C\e^{-m|t|},\quad |\b(t)|\leq C\e^{-m|t|}, \quad \text{for all}\quad t\in \R.
\end{equation}
For $\eta>0$ denote by $\mathcal{K}_{\eta}$ the trajectory of \eqref{toy0} given by the  initial conditions $a_{0}(0)=\eta$, $b_{0}(0)=0$ (see Figure \ref{PhasePortraita0b0}).  In our forthcoming paper \cite{GJT}, we study the long time stability of the  trajectory   $\mathcal{K}_{\eta}$ for $0<\eta\ll1$, by the flow of the system \eqref{syst}.
 
\begin{figure}[h!]
\centering 
\def\svgwidth{100mm}
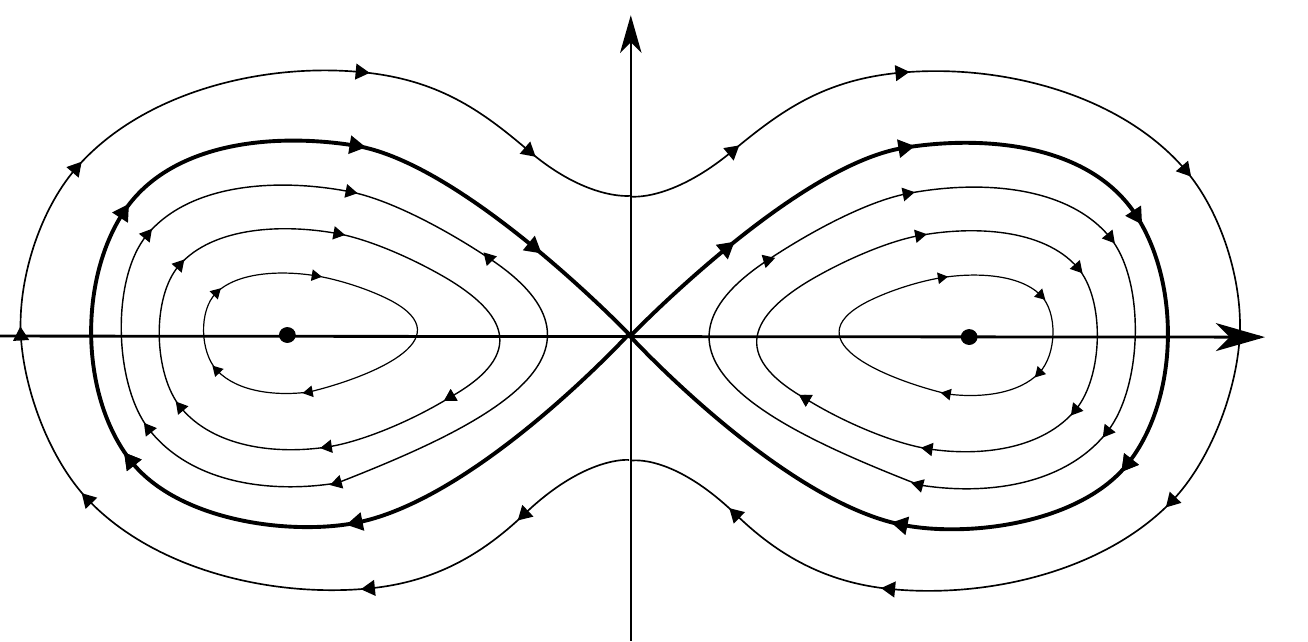
\caption{Phase portrait for the space-stationary set $a_{n}=b_{n}=0$ for $n\geq 1$.}
\label{PhasePortraita0b0}
\end{figure}

\subsection{The main results on the Klein-Gordon equation}
Under Assumptions \ref{assumption0} and \ref{assumption},  the equation \eqref{kg} is globally well-posed in $\X=H^{1}(M)\times L^{2}(M)$ (see Theorem \ref{thm0} in Section \ref{Sect.2}). \ligne

Our main  result describes possible dynamics near the homoclinic orbit: We show that there are many solutions which stay close to $\h$. The proof of this result is  inspired from the work \cite{GS1} of M. Groves and G. Schneider. In our case, the novelty is that we have not to assume $m$ to be small, {\it i.e.} we are not dealing with small solutions of \eqref{kg}.\\
For any function $f\in L^{2}(M)$, denote by $\dis f_{c}=f-\int_{M}f$. Then we can state
\begin{theo}\ph\label{thm1}
Let $M$, $p$ and $m$ satisfy Assumptions \ref{assumption0} and \ref{assumption}. There exist $C>0$ and $\eps_{0}>0$ so that  the following holds: Let $(f,g)\in H^{1}(M)\times L^{2}(M)$ be such that $\|f_{c}\|_{H^{1}}+\|g_{c}\|_{L^{2}}\leq \eps_{0}$, then there exists a solution $u$ to \eqref{kg} so that 
\begin{equation}\label{sol1}
\partial_{t}u(0,\cdot)=g_{c},\quad u_{c}(0,\cdot)=f_{c},
\end{equation}
and so that for all $t\geq 0$
\begin{equation}\label{sol2}
\|u(t,\cdot)-\a(t)\|_{H^{1}}+\|\partial_{t}u(t,\cdot)-\b(t)\|_{L^{2}}\leq C\big(\|f_{c}\|_{H^{1}}+\|g_{c}\|_{L^{2}}\big).
\end{equation}
Furthermore, as $t\longrightarrow + \infty$, $u(t,\cdot)$ tends to a solution of the local center manifold $W^c$ (see Definition~\ref{def46}).
\end{theo}
Notice that Theorem \ref{thm1} is not a stability result since we only fix the elliptic part of the initial datum, and the statement of the theorem is that there exists a way to define the hyperbolic part such that the solution satsifies \eqref{sol2}.  For example if $f_{c}=g_{c}=0$, the corresponding $u$ is the homoclinic orbit $h$.

The equation \eqref{kg} is reversible. More precisely, 
 the equation \eqref{kg} is invariant under the transformation $u(t,x)\longmapsto u(-t,x)$. This corresponds to the symmetry $S$ defined by 
  \begin{equation*}
  S(a_{n})=a_{n}\quad   S(b_{n})=-b_{n},\quad \text{for} \quad n\geq 0.
  \end{equation*}
 This is  a symmetry of reversibility: Namely, denote by $V_{H}$ the vector field \eqref{syst}, then $V_{H}\circ S=-S\circ V_{H}$.\ligne

Using this symmetry we are able to precise the result of Theorem \ref{thm1} in the case $g=0$.

\begin{theo}\ph\label{coro.thm}
Let $M$, $p$ and $m$ satisfy Assumptions \ref{assumption0} and \ref{assumption}. There exist $C>0$ and $\eps_{0}>0$ so that  the following holds: Let $f \in H^{1}(M)$ be  such that $\|f_{c}\|_{H^{1}}\leq \eps_{0}$, then there exists a solution $u$ to~\eqref{kg} so that 
\begin{equation} \label{sol3}
\partial_{t}u(0,\cdot)=0,\quad u_{c}(0,\cdot)=f_{c},
\end{equation}
for all $t\in \R$ 
$$u(-t)=Su(t),$$
 and so that 
\begin{equation*} 
\|u(t,\cdot)-\a(t)\|_{H^{1}}+\|\partial_{t}u(t,\cdot)-\b(t)\|_{L^{2}}\leq C\|f_{c}\|_{H^{1}}.
\end{equation*}
Furthermore, as $t\longrightarrow + \infty$, $u(t,\cdot)$ tends to a solution of the local center manifold $W^c$.
\end{theo}

Since $u$ is reversible, observe that if $u$ tends to $w$ when $t\longrightarrow +\infty$, then $u$ tends to $Sw$ when $t\longrightarrow -\infty$. 

\begin{rema} Theorem \ref{coro.thm} proves the existence of heteroclinic orbits from some orbit on the center manifold $W^c$ to another one. It would be interesting to precise the dynamics on the center manifold. Actually we can expect the existence of invariant tori on $W^c$ which are perturbations of the tori obtained for the linearized equation around the homoclinic orbit (see section 3). Nevertheless, to prove this, we would need a KAM theory for infinite dimensional tori which is an open problem up to now (see however \cite{Bour}). Notice that even if such theory would exist, it would prove that most of the tori, but not all of them, of the linearized system are preserved. Thus it would be  impossible to conclude   that  our heteroclinic orbits tends to KAM tori of $W^c$  when $t\to \pm\infty$ as we can expect.  \end{rema}

  \subsection{Plan of the paper}
 In Section \ref{Sect.2} we prove Theorem \ref{thm0}. Section \ref{Sect.3} is devoted to the study of the linearisation of \eqref{kg} around the homoclinic orbit (Theorem \ref{thm3}). In Section \ref{Sect.4} we construct the center manifold for the Klein-Gordon equation and prove Theorem \ref{thm1}. In the appendix, we recall a result which is useful in the proof of Theorem \ref{thm3}.

  \begin{enonce*}{Notations}
 In this paper $c,C>0$ denote constants the value of which may change
from line to line. These constants will always be universal, or depend on the fixed quantities $m$ and $p$.\\
We denote by $\N$ the set of the non negative integers, and $\N^{*}=\N\backslash\{0\}$. We set $\X=H^{1}(M)\times L^{2}(M)$.
 \end{enonce*}
 
 \begin{acknowledgements}
The authors want to thank Eric Lombardi and Romain Joly for clarifications in dynamical systems. T.J. also thanks Guido Schneider for his kind invitation to Stuttgart.
\end{acknowledgements}

\section{Preliminaries: Proof of the first results on (KG)}\label{Sect.2}

\subsection{Stationary solutions of \eqref{kg}}

We look for the time-independent solutions to \eqref{kg}, which correspond to equilibrium points for the system \eqref{syst}.
\begin{lemm}\ph\label{lem.equi}
The stationary solutions to  \eqref{kg} are $u=0$, $u=m^{1/p}$ and  $u=-m^{1/p}$. 
\end{lemm}

\begin{proof}
Assume that $u(t,x)=g(x)\in H^{1}(M)$ is solution to \eqref{kg}. We apply $\nabla$ to the equation and get
\begin{equation*}
-\Delta \nabla g-m^{2}\nabla g+(2p+1)g^{2p}\nabla g=0.
\end{equation*}
We multiply this equation with $\nabla g$ and integrate (by parts) on $M$
\begin{equation}\label{g}
\int_{M}\big(\Delta g\big)^{2}+(2p+1)\int_{M}g^{2p}|\nabla g|^{2}=m^{2}\int_{M}|\nabla g|^{2}.
\end{equation}
Now, write $\dis g=\sum_{k=0}^{\infty} g_{k}e_{k}$. Since $(\lambda_{k})_{k\geq 0}$ is non decreasing and $0<m<\lambda_{1}$ we have 
\begin{equation*}
m^{2}\int_{M}|\nabla g|^{2}=m^{2}\sum_{k=1}^{\infty}\lambda^{2}_{k}g^{2}_{k}<  \sum_{k=1}^{\infty}\lambda^{4}_{k}g^{2}_{k}= \int_{M}\big(\Delta g\big)^{2}.
\end{equation*}
This inequality together with \eqref{g} implies that $\nabla g=0$, and we conclude.
\end{proof}

\subsection{Global well-posedness of \eqref{kg}}
The aim of this section is to prove the following result.
\begin{theo}\ph\label{thm0}
Let $M$, $p$ and $m$ satisfy Assumptions \ref{assumption0} and \ref{assumption}. Then for all $(u_{0},u_{1})\in \X$ there exists a unique solution to \eqref{kg}
\begin{equation*}
u\in \mathcal{C}^{0}\big(\R;H^{1}(M)\big)\cap \mathcal{C}^{1}\big(\R;L^{2}(M)\big).
\end{equation*}
Moreover, $u$ is bounded in the energy space: $\dis \sup_{t\in \R}\|(u,\partial_{t}u)\|_{\X}\leq  C$.
\end{theo}

The proof is classical: it relies on a fixed point argument in $H^{1}(M)$ combined with the conservation of the energy.\ligne

To begin with, for $\dis u=\sum_{k=0}^{\infty}a_{k}e_{k}$ and $\dis \partial_{t}u=\sum_{k=1}^{\infty}b_{k}e_{k}$ we set 
 \begin{equation}\label{defJ}
 J:=\frac12\sum_{k=1}^{\infty}\Big[(\lambda_{k}^{2}-m^{2})a^{2}_{k}+b^{2}_{k}\Big].
 \end{equation}
We also define  
\begin{equation}\label{defU}
 U:=u_{c}= \sum_{k=1}^{\infty}a_{k}e_{k}
\end{equation}
the spectral projection away the mode 0. 
 Observe that under Assumption \ref{assumption0} there exist $C_{1},C_{2}>0$ so that for all $(u,\partial_{t}u)\in \X$
 \begin{equation}
 C_{1}\|u_{c},\partial_{t}u_{c}\|_{\X}\leq J^{1/2}\leq C_{2}\|u_{c},\partial_{t}u_{c}\|_{\X},
 \end{equation}
 because $ \|u_{c},\partial_{t}u_{c}\|_{\X} =\|u_{c}\|_{H^{1}}+\|\partial_{t}u_{c}\|_{L^{2}} $ and 
 $$  \|u_{c}\|^{2}_{H^{1}}+\|\partial_{t}u_{c}\|^{2}_{L^{2}}  = \sum_{k=1}^{\infty}\big(\lambda_{k}^{2} a^{2}_{k}+b^{2}_{k}\big).$$
 
In the sequel we will need

\begin{lemm}\ph\label{lemsobo}
Let $u\in H^{1}(M)$. Then  \\[5pt]
$\bullet$ When $M$ has dimension 1 or 2, for all $2\leq q<+\infty$
\begin{equation*}
\|U\|_{L^{q}(M)}\leq C_{q} J^{1/2}.
\end{equation*}
$\bullet$ When $M$ has dimension 3, for all $2\leq q\leq 6$
\begin{equation*}
\|U\|_{L^{q}(M)}\leq C_{q} J^{1/2}.
\end{equation*}
 \end{lemm}
 
\begin{proof}
By Sobolev, in each of the previous cases, there exists $C_{q}>0$  so that for all $U\in H^{1}(M)$ we have $\|U\|_{L^{q}(M)}\leq C_{q} \|U\|_{H^{1}(M)}$ and the result follows.
\end{proof}

We then can prove the following a priori estimate
\begin{lemm}\ph\label{bounded} Let $u(t,\cdot)$ be a solution of \eqref{kg} and 
denote by $H^{0}$ the constant value of $H(u(t,\cdot))$ along this trajectory. For all $t$ for which the solution is defined, we have
\begin{equation}\label{4.8}
\int_{M}\big((\partial_{t}u)^{2}+|\nabla u|^{2}\big)=b^{2}_{0}+\sum_{k=1}^{\infty}\big(b^{2}_{k}+\lambda_{k}^{2}a^{2}_{k}\big)\leq 2H^{0}+\frac{p}{p+1}m^{2+2/p}.
\end{equation}
Furthermore there exists $C\equiv C(H^{0})>0$ such that $|a_0(t)|\leq  C$ for all $t$.
\end{lemm}
 
 \begin{proof} We assume that $M$ has dimension 1 or 2.
 By the conservation of the energy for \eqref{kg}, we have
 \begin{equation*}
 \int_{M}\big((\partial_{t}u)^{2}+|\nabla u|^{2}\big)+\frac{1}{p+1}\int_{M}u^{2p+2}=2H^{0}+m^{2}\int_{M}u^{2}.
 \end{equation*}
We apply  the Young inequality
\begin{equation*} 
c_{1}c_{2}=(\eps c_{1})(\eps^{-1}c_{2})\leq \frac{\eps^{q}}{q}c_{1}^{q}+\frac{1}{r\eps^{r}}c_{2}^{r},\quad c_{1},c_{2}\geq 0,\quad \eps>0, \quad \frac1{q}+\frac1{r}=1,
\end{equation*} 
with $c_{1}=u^{2}$, $c_{2}=1$, $q=p+1$ and $\eps=m^{-2/(p+1)}$ and deduce \eqref{4.8}. 

Recall the notations \eqref{defJ} and \eqref{defU}. Then   we have
\begin{equation}\label{H0}
H^0=\frac12(b^{2}_{0}-m^{2}a^{2}_{0})+J+\frac1{2(p+1)}\int_{M}\big(a_{0}+U\big)^{2p+2}\text{d}x.
\end{equation} 
By \eqref{4.8}, $b_0$ and $J$ are bounded and  moreover by Lemma \ref{lemsobo} and \eqref{4.8} we deduce that $\int_{M}|U|^q \text{d}x\leq CJ^{q/2}$ for all $q\geq 2$.  Thus \eqref{H0} gives $H^0$ as a polynomial in $a_{0}$ with bounded coefficients and thus $a_0$ has to be  bounded by a constant depending only on $H^0$. \\
The argument is similar when $M$ has dimension 3 and $p=1$.
\end{proof}
 
 \begin{rema}
 It is easy to check that the points $u=\pm m^{1/p}$ correspond to the minimum of $H$, and that for this choice $\dis H^{0}=-\frac{p}{2(p+1)}m^{2+2/p}$. In particular, the r.h.s of \eqref{4.8} is nonnegative.
 \end{rema}

\begin{proof}[Proof of Theorem \ref{thm0}]
First we show that \eqref{kg} is locally well-posed. Let $(u_{0},u_{1})\in \X$ and for~$T>0$ define the space 
\begin{equation*}
E_{T}=\big\{\,u\in \mathcal{C}^{0}\big([-T,T];H^{1}(M)\big)\cap \mathcal{C}^{1}\big([-T,T];L^{2}(M)\big),\;\; 
\sup_{|t|\leq T}\|(u,\partial_{t}u)\|_{\X}\leq 2\|(u_{0},u_{1})\|_{\X}\,\big\}.
\end{equation*}
 For $f:\R\longrightarrow \R$, a continuous and bounded function, we can define  a bounded operator 
 $$f(-\Delta) :L^{2}(M)\longrightarrow L^{2}(M)$$ by 
 $$ f(-\Delta)u=\sum _{k=0}^{+\infty}f(\lambda^{2}_{k}) a_{k}e_{k},\quad  u=\sum _{k=0}^{+\infty} a_{k}e_{k}.$$
 With this definition, denote by $K(t)$ the  free wave propagator  
 \begin{equation*}
 K(t)(u_{0},u_{1})=\cos(t\sqrt{-\Delta} )u_{0}+ \frac{\sin( t\sqrt{-\Delta})}{\sqrt{-\Delta}}u_{1}.
 \end{equation*}
  Then we show that the mapping $\Phi$ defined by 
\begin{equation*}
\Phi(t)u=K(t)(u_{0},u_{1})+\int_{0}^{t}\frac{\sin(t-s)\sqrt{-\Delta}}{\sqrt{-\Delta}}\big(m^{2}u-u^{2p+1}\big)(s)\text{d}s,
\end{equation*}
is a contraction of $E_{T}$ for $T$ small enough. Namely, by the Sobolev embeddings and the fact that $\|K\|_{L(\X,\X)}\leq C$, we obtain that we can take  $T=c_{0}\min\big(m^{-2},\|(u_{0},u_{1})\|^{-2p}_{\X}\big)$ where $c_{0}>0$ is a small absolute constant.\\
In order to prove the global well-posedness, we iterate the previous argument. We then obtain a sequence of times $T_{n}$ with $T_{n+1}=c_{0}\min\big(m^{-2},\|(u(T_{n}),\partial_{t}u(T_{n})\|^{-2p}_{\X}\big)$ such that the solution is defined on $(-\sum_{n=0}^\infty T_n, \sum_{n=0}^\infty T_n)$.  By Lemma \ref{bounded}, for any $t>0$, $\|(u(t),\partial_{t}u(t))\|^{-2p}_{\X}\geq C(H^{0})$, hence $\sum_{n\geq 0}T_{n}=+\infty$ which gives the result. 
\end{proof}
\begin{rema}
The result of Theorem \ref{thm0} indeed holds  for any $m\in \R$.
\end{rema}




 
\section{The linearized equation around the homoclinic orbit}\label{Sect.3}

In this section, we study the system \begin{equation} \label{toy1}
\left\{
\begin{aligned}
&\dot{a}_{0} =b_{0}\\
&\dot{b}_{0} =m^{2}a_{0}- a^{2p+1}_{0},
\end{aligned}
\right.
\end{equation} 
and for $n\geq 1$
\begin{equation} \label{toy2}
\left\{
\begin{aligned}
&\dot{a}_{n} =b_{n}\\
&\dot{b}_{n} =-(\lambda_{n}^2-m^{2})a_{n}-(2p+1)a^{2p}_0a_n.
\end{aligned}
\right.
\end{equation} 
Indeed, if one considers that $a_{0}$ is given by \eqref{toy1}, the system \eqref{toy2} is the linearisation of \eqref{kg} around the solution $u=a_{0}$. In particular, denoting $V(t)=(2p+1)a^{2p}_{0}(t)$ and $w=\sum_{k\geq 1}a_{k}e_{k}$, \eqref{toy2} is equivalent to the following linear wave equation with time-dependent potential
   \begin{equation}\label{linearwave}
   \partial^{2}_{t}w-\Delta w-m^{2}w+V(t)w=0.
   \end{equation}
Notice that the system \eqref{toy1}, \eqref{toy2} is not Hamiltonian for the canonical structure, but it is reversible. On the other hand, for $a_0$ given, the system \eqref{toy2} is Hamiltonian for the canonical structure. \ligne

We now study  the linear evolution of \eqref{toy2}, which will be useful in the proof of Theorem \ref{thm1}.
Introduce the notations
 \begin{equation*} 
Z=\left(\begin{array}{c}
z\\\partial_{t}z\end{array}\right) ,\quad \mathcal{L}^{t}=
\left(\begin{array}{cc}
0& 1\\\Delta+m^{2}-(2p+1)\a^{2p}(t)& 0\end{array}\right),
\end{equation*}
then \eqref{toy1}, \eqref{toy2} can also be written 
\begin{equation*}
\partial_{t}Z=\L^{t}Z.
\end{equation*}

\begin{prop}\ph\label{Prop.Lin}
The  restriction $\mathcal{L}^{t}_{c} : \mathcal{X}_{c} \longrightarrow \mathcal{X}_{c}$ generates a two parameter group $K(t,\tau)$ so that 
\begin{equation*}
\sup_{t,\tau \in \R}\|K(t,\tau)\|_{L(\mathcal{X}_{c},\mathcal{X}_{c})}\leq C.
\end{equation*}
In other words, every solution to $\dis \partial_{t}Z_{c}=\L^{t}_{c}Z_{c}$ is bounded.
\end{prop}

Proposition \ref{Prop.Lin} is  a direct consequence of  the next result.  Define the tori  by
$$
\mathcal T_c:=\big\{(a_k,b_k)_{k\geq 1}\mid (\lambda_{n}^2-m^2)a^2_n+b_n^2=c_n^2\;\big\},\quad c\in \l^2(\N,\R),
$$
then we have

 \begin{theo}\ph\label{thm3}
 Let $(a_{0}(0),b_{0}(0))\in \mathcal{K}_{0} $ and $(a_{n}(0),b_{n}(0))_{n\geq 1}\in h^{1}\times \ell^{2}$. Then the corresponding solution of \eqref{toy1}, \eqref{toy2}    is homoclinic to $\mathcal T_c$ for some  $c\in \l^2(\N,\R)$.  
  \end{theo}
   This result can also be interpreted as a linear scattering result for the system \eqref{toy2}. The infinite dimensional system \eqref{toy1}, \eqref{toy2}  has homoclinic orbits to invariant tori of arbitrary large dimension (finite or infinite). These tori correspond to the case $a_0=b_0=0$ which is stable by~\eqref{toy1}. In that case \eqref{toy2} becomes the standard harmonic oscillator in infinite dimension with frequency vector  $\omega(m)=(\omega_n)_{n\geq 1}$ and $\omega_n=\sqrt{\lambda_{n}^2-m^2}$. 

    \proof
Take the homoclinic $(\alpha,\beta)$  in the plane  $(a_0,b_0)$,  and observe that   $\dis \int_\R \alpha^{2p}(t)\text{d}t<\infty$. Now, for $n\geq 1$ we introduce the norms on $\R^2$, $\dis |x,y|_{n}=\big((\lambda_{n}^2-m^2)x^2+y^2\big)^{1/2}$ and the norm on $h^1\times\l^2$ $\dis \big\|(a_{n},b_{n})_{n\geq 1}\big\|=\big(\sum_{n\geq 1}|a_{n},b_{n}|^{2}_{n}\big)^{1/2}$. 
Then for all $n\geq 1$, we deduce from Lemma \ref{A2} that the trajectory $(a_{n},b_{n})$ satisfies $\dis |a_{n},b_{n}|_{n}\leq C|a_{n}(0),b_{n}(0)|_{n}$. Next, we denote by $S$ the flow of the linear part of \eqref{toy2}, {\it i.e.} with $a_{0}=0$. Then the solution of \eqref{toy2} reads
\begin{equation}\label{duha}
\left(\begin{array}{c}
a_{n}(t)\\b_{n}(t)\end{array}\right) =S(t)\left(\begin{array}{c}
a_{n}(0)\\b_{n}(0)\end{array}\right)-(2p+1)\int_0^t S(t-s)\left(\begin{array}{c}
0\\ \a^{2p}(s)a_{n}(s)\end{array}\right)\text{d}s.
\end{equation}
By construction $S$ preserves the norm $\big |\;.\;\big|_{n}$ and thus   
\begin{eqnarray}
\int_0^{+\infty} \Big|S(-s)\left(\begin{array}{c}
0\\ \a^{2p}(s)a_{n}(s)\end{array}\right)\Big|_{n}\text{d}s &=&\int_0^{+\infty} \a^{2p}(s)|a_{n}(s)|\text{d}s\nonumber\\
&\leq & C|a_{n}(0),b_{n}(0)|_{n}\label{integral},
\end{eqnarray}
where we used the boundedness of $a_{n}$ and the integrability of $\a^{2p}$. Denote by 
$$
\left(\begin{array}{c}
a^{+}_{n}\\b^{+}_{n}\end{array}\right) =\left(\begin{array}{c}
a_{n}(0)\\b_{n}(0)\end{array}\right)-(2p+1)\int_0^{+\infty}S(-s)\left(\begin{array}{c}
0\\ \a^{2p}(s)a_{n}(s)\end{array}\right)\text{d}s,$$
therefore by \eqref{duha} we get 
\begin{eqnarray*} 
\Big|\left(\begin{array}{c}
a_{n}(t)\\b_{n}(t)\end{array}\right) -S(t)\left(\begin{array}{c}
a^{+}_{n}\\b^{+}_{n}\end{array}\right)\Big|_{n}&\leq&(2p+1) \int_t^{+\infty} \a^{2p}(s)|a_{n}(s)|\text{d}s\\
&\leq &C |a_{n}(0),b_{n}(0)|_{n} \int_t^{+\infty} \a^{2p}(s)\text{d}s,
\end{eqnarray*}
which in turn implies
\begin{equation*}
\Big\|\left(\begin{array}{c}
a_{n}(t)\\b_{n}(t)\end{array}\right) -S(t)\left(\begin{array}{c}
a^{+}_{n}\\b^{+}_{n}\end{array}\right)\Big\| \longrightarrow 0, \quad \text{when}\;\;t\longrightarrow +\infty.
\end{equation*}
Finally we remark that the trajectory $t\mapsto S(t)\left(\begin{array}{c}
a^{+}_{n}\\b^{+}_{n}\end{array}\right)$ lives on the torus $\mathcal T_c $ with $ c_n^2=(\lambda_{n}^2-m^2)(a^+_n)^2+(b_n^+)^2$ for $n\in\N$.
 \endproof
 
 \begin{rema}
If $\om(m)$ is non resonant (for example when $M=\S^{1}$, this happens for a generic  choice of the mass $m$, see \cite{BG06}), $S(t)\big(a^{+}_{n},b_{n}^{+}\big)_{t\geq 0}$ densely covers the torus $\mathcal{T}_{c}$. 
\end{rema}
 \begin{rema} Notice that in the proof of Theorem \ref{thm3} we crucially use that $a_0=\alpha$, {\it i.e.} that we  linearize around the homoclinic orbit. 
A natural question is to ask whether this result still holds if $(a_{0},b_{0})$ is any periodic solution of \eqref{toy1}. In that case $\int_\R a_0^{2p}(t) \text{d}t= +\infty$ and Lemma \ref{A2} does not apply. In particular it is not clear at all that the trajectories remain bounded. This seems to be a difficult and interesting problem related to the reducibility of \eqref{linearwave} when $V$ is not a priori small (the case $V$ small and quasi-periodic in time can be solved by a KAM approach, see \cite{EK, GT} for the last results in the Schr\"odinger case).
\end{rema}

Denote by $\L^{t}_{h}:= \L^{t}_{|\X_{h}}=\left(\begin{array}{cc}
0& 1\\m^{2}-(2p+1)\a^{2p}(t)& 0\end{array}\right)  $ the restriction of $\L^{t}$ to $\X_{h}$. We can identify the space $\X_{h}=\R^{2}$ and we denote by $\<\,,\,\>$ the Euclidian scalar product. Then 
\begin{lemm}\ph \label{LemLin}
The equation 
\begin{equation*}
\partial_{t}Z_{h}=\L^{t}_{h}Z_{h},
\end{equation*}
admits two solutions $\s(t)$ and $\rho(t)$ defined on $\R_{+}$ so that
\begin{equation}\label{init}
\s(0)=\left(\begin{array}{cc}
0 \\ 1 \end{array}\right), \quad \rho(0)=\left(\begin{array}{cc}
1 \\ 0 \end{array}\right),
\end{equation}
(see Figure \ref{Linearizeh}) and
\begin{equation}\label{BB1}
|\s(t)|\leq C \e^{-mt},\quad |\rho(t)|\leq C \e^{mt},\quad \text{for all}\quad  t\geq 0.
\end{equation}
The dual basis $\{\s^{\star},\rho^{\star}\}$ to $\{\s,\rho\}$ in $\X_{h}$ satisfies
\begin{equation}\label{BB2}
|\s^{\star}(t)|\leq C \e^{mt},\quad |\rho^{\star}(t)|\leq C \e^{-mt},\quad \text{for all}\quad  t\geq 0.
\end{equation}
\end{lemm}

\begin{figure}[h!]
\centering 
\def\svgwidth{100mm}
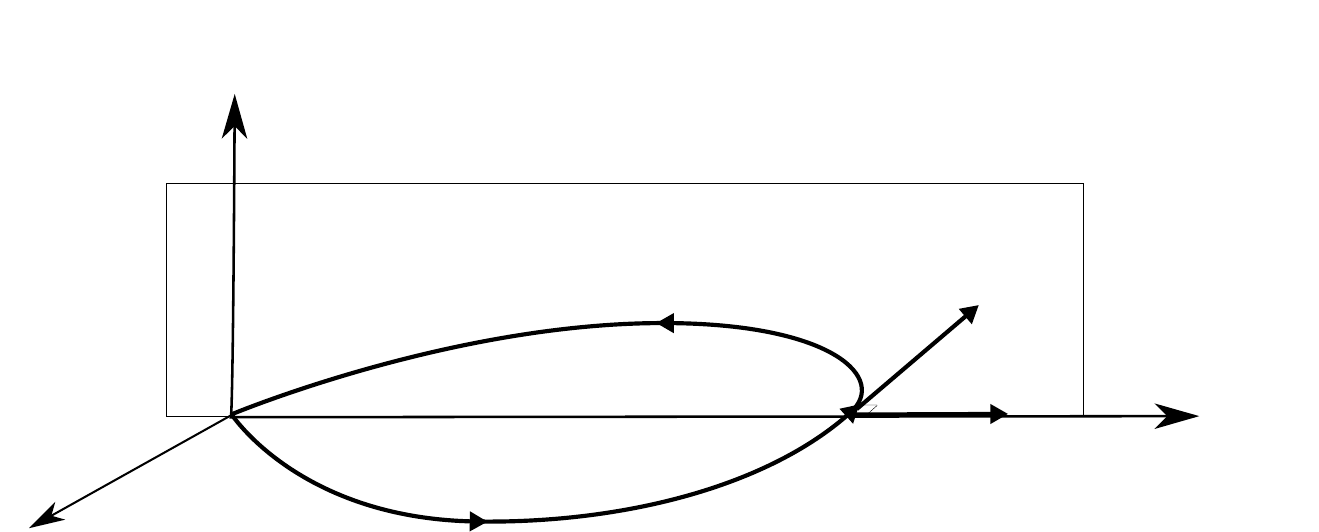
\caption{Linearized system at $t=0$ and reversibility plane.}
\label{Linearizeh}
\end{figure}

\begin{proof}
Let $\mu\in \R$, then $\dis t\mapsto \mu \left(\begin{array}{cc}
\dot{\a}(t)\\ \ddot{\a}(t)\end{array}\right)=\mu\left(\begin{array}{cc}
\b(t)\\ \dot{\b}(t)\end{array}\right)$ is a solution of $\partial_{t}Z_{h}=\L^{t}_{h}Z_{h}$. Then in view of \eqref{def.homo} we  deduce that this solution corresponds to $\s$ and we obtain the first bound in \eqref{BB1}. With a suitable choice of $\mu$ we get \eqref{init}.

To find $\rho$, we write   $\dis \rho= \left(\begin{array}{cc}
{\gamma}\\ \dot{\gamma}\end{array}\right)$, and look for $\gamma$ of the form $\gamma=\dot{\a}z$, when $t\geq 1$. Then $z$ satisfies the equation $\dot{\a}\ddot{z}+2\ddot{\a}\dot{z}=0$ and therefore 
\begin{equation*}
z(t)=z(1)+\dot{z}(1)(\dot{\a}(1)\big)^{2}\int_{1}^{t}\big(\dot{\a}(\tau)\big)^{-2}\text{d}\tau,\quad \forall t\geq 1.
\end{equation*}
It is then straightforward to check that $\gamma=\dot{\a}z$ can be extended to a $\mathcal{C}^{\infty}$ function on $\R$.
Then the bound on $\s$  gives $|z(t)|\leq C\e^{2mt}$ and $|\dot{z}(t)|\leq C\e^{2mt}$ for all $t\geq 0$, which in turn implies that $|\gamma(t)|\leq C\e^{mt}$ and $|\dot{\gamma}(t)|\leq C\e^{mt}$ for all $t\geq 1$. Hence the second bound in \eqref{BB1}. We can moreover choose the initial value $\rho$ of the form claimed in the statement of the lemma.\ligne

Then we can explicitly compute 
\begin{equation*}
\s^{\star}=\frac1{\b \dot{\gamma}-\gamma \dot{\b}}\left(\begin{array}{cc}
\dot{\gamma}\\ -{\gamma}\end{array}\right),\quad \rho^{\star}=\frac1{\b \dot{\gamma}-\gamma \dot{\b}}\left(\begin{array}{cc}
-\dot{\b}\\ {\b}\end{array}\right),
\end{equation*}
 where $\b \dot{\gamma}-\gamma \dot{\b}$ is a non vanishing constant, since it is the Wronskian of the equation 
 $$\dis \ddot{y}-m^{2}y+(2p+1)\a^{2p}y=0.$$
  The bound \eqref{BB2} then follows from \eqref{BB1}.
\end{proof}

\section{Homoclinic orbits to the center manifold}\label{Sect.4}

In this section we prove Theorem \ref{thm1}.
As we mentioned in the introduction, this part closely follows the argument of Groves-Schneider  \cite{GS3}. Nevertheless in contrast to their work, we are not dealing with~$m$ small and thus with small solutions. Even if some results already appear in \cite{GS3}, we reproduce here all the proofs for the convenience of the reader. First we recall the general strategy, illustrated by Figure \ref{Schemastrategie}.
\begin{enumerate}
\item To begin with, we consider a spatially truncated system for which the non-linearity sees only the elliptic (or central) modes of small amplitude. We construct global solutions to this system that are close to the homoclinic orbit, namely the hyperbolic part is close to $h=(\a,\b)$ for all times $t\geq 0$, and the elliptic part is small only for $t\leq \eps^{-1}$ where $\eps$ is a small parameter. (Subsection~\ref{4.2}) \\
In the next steps we will prove that the elliptic part of these solutions actually stay small for all~$t$, $i.e.$ that the set of all these solutions is the global center-stable manifold $W^{cs}$ of the truncated system, and thus a local center-stable manifold for \eqref{kg} at time $t=0$. 
\item We construct a global center manifold $W^{c}$ for the truncated system by constructing the solutions that are close to the origin for all times. 
(Subsection~\ref{4.3})
\item The crucial point consists in linking the two previous steps: we prove that the solutions of the truncated system constructed in step 1 actually tend to some solutions contained in the local center manifold as $t\to +\infty$ and thus remain small for all times $t\geq0$. As the central part remains unconditionally small, we conclude that these solutions are actually solutions of the original system (without truncation). This gives us a global center stable manifold $\tilde{W}^{cs}$ which is parametrized by the initial value $V_s$ of the stable part that we add to $h$ and by $V_c$ the initial value of the elliptic part, both being small. (Subsection \ref{4.4})
\item For the proof of Theorem \ref{coro.thm}, it remains to prove that we can choose $V_s$ and $V_c$ in such a way that the solution hits the reversibility plane $b_n=0$ for all $n$: namely the solution at $t=0$ belongs to the reversibility plane as soon as $V_s=0$, and $V_c$ is symmetric. The corresponding solutions are then automatically symmetric and thus describe heteroclinic connections between two symmetric solutions of the center manifold. (Subsection \ref{4.5})
\end{enumerate}

 \begin{figure}[h!]
\centering 
\def\svgwidth{130mm}
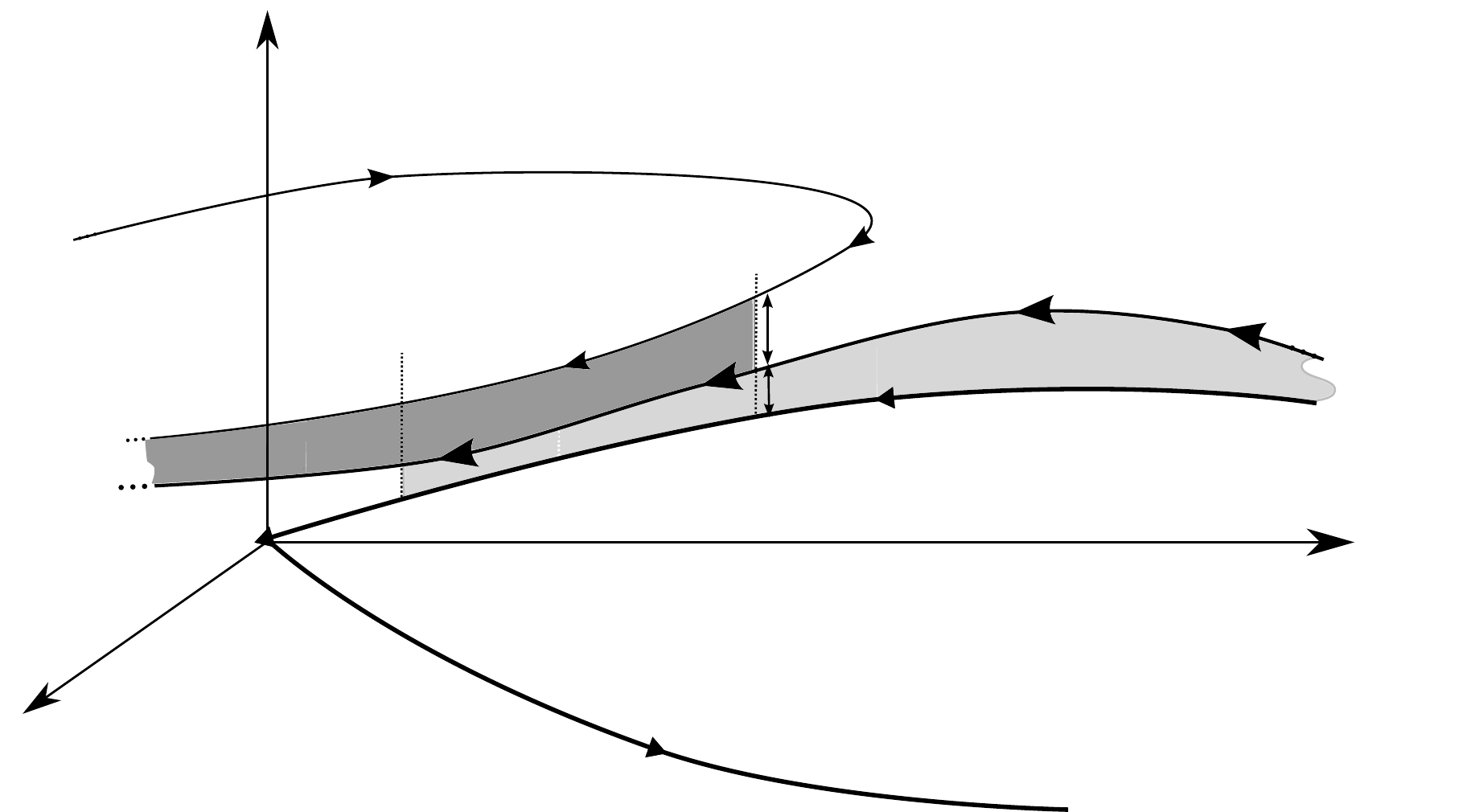
\caption{Schematic overview of the strategy and the notations. We will see later that we take $\delta=\eps^{3/2}$ with $\eps\ll 1$.}
\label{Schemastrategie}
\end{figure}

\subsection{Preliminary results}
To begin with, we  write an  equivalent formulation of \eqref{kg}. Denote by 
\begin{equation} \label{def.F}
X= \left(\begin{array}{cc}
u\\ v\end{array}\right),  \quad   {\Lambda}=
\left(\begin{array}{cc}
0& 1\\ \Delta+m^{2}& 0\end{array}\right)  ,\quad F(X)=\left(\begin{array}{c}
0\\ -u^{2p+1}\end{array}\right).  
\end{equation}
The equation \eqref{kg} is equivalent to  
\begin{equation}\label{eqme}
\partial_{t}X=\Lambda X+F(X).
\end{equation}
By the result of Theorem \ref{thm0}, the equation \eqref{eqme} admits, for a given initial condition $X(0)\in \X$ a unique global solution $X  \in \X$ so that $\dis \sup_{t\in \R}\|X(t)\|\leq C(\|X(0)\|_{\X})$. We fix $C_{0}>0$ large enough, so that all the initial conditions $X(0)$ we consider in the sequel satisfy $\|X(0)\|\leq C_{0}$. Therefore all the solutions  $X$ we are going to consider satisfy
$$\dis \sup_{t\in \R}\|X(t)\|\leq C(C_{0}).$$
We denote  $\D_c$ the projection on $\X_c$ of such initial conditions:
$$\D_c:=\big\{\,Q X\mid \|X\|\leq C_0\,\big\}.$$

Let $\theta\in \mathcal{C}^{\infty}([0,\infty),\R)$  be a cut-off function  so that
\begin{equation*}  
\theta(s)=\left\{
\begin{aligned}
& 1,\quad   s\leq \delta\\
&  0,\quad s\geq 2\delta,
\end{aligned}
\right.
\end{equation*} 
where $\delta$ is a  parameter much  smaller than 1.\\
For any function $G: \X\longrightarrow \R$, denote by 
\begin{equation*}
\underline{G}(X)=G\big(X\theta(\|X_{c}\|)\big).
\end{equation*} 

\begin{lemm}\ph\label{4.0}
Consider the function $F$ given by \eqref{def.F} and let $X,X'\in \mathcal{X}$. 
Assume moreover  that $|X_{h}|,|X'_{h}|\leq \delta<1$. Then there exists a constant $C>0$ such that
\begin{equation}\label{Ftronca0}
\|\underline{{F}}(X)\|\leq  C\delta^{3} , 
\end{equation}
\begin{equation}\label{Ftronca}
\|\underline{{F}}(X)-\underline{{F}}(X')\|\leq  C\delta^{2} \|X-X'\|.
\end{equation}
\end{lemm}

\begin{proof}

Writing $X=\left(\begin{array}{cc}
u\\ v\end{array}\right)$ and $u=a+U$ with $a\in \X_h$ and $U\in \X_c$ we have by the fact that $|a|\leq \delta$ and the Sobolev embedding 
\begin{eqnarray*}
\|\underline{{F}}(X)\|&\leq & \big\|\big(|a|+|U|\theta(\|U\|_{H^{1}})\big)^{2p+1}\big\|_{L^{2}(M)}\\
&\leq & C\delta^{2p+1}+C \|U\|^{2p+1}_{L^{2(2p+1)}}\theta(\|U\|_{H^{1}})\\
&\leq & C\delta^{2p+1}+C \|U\|^{2p+1}_{H^{1}}\theta(\|U\|_{H^{1}}).
\end{eqnarray*}
This gives  \eqref{Ftronca0}, since $\|U\|_{H^1}  \theta(\|U\|_{H^{1}})\leq C\delta$.\\
To get \eqref{Ftronca}, just remark that for $|u|,|u'|\leq C\delta$
$$
\Big\|F\left(\begin{array}{cc}
u\\ v\end{array}\right)-F\left(\begin{array}{cc}
u'\\ v'\end{array}\right)\Big\|^2\leq \int_{M}\big|u^{2p+1}-u'^{2p+1}\big|^2\text{d}x\leq  C\delta^4\int_{M}|u-u'|^2\text{d}x,$$
and then apply this inequality to $\left(\begin{array}{cc}
u\\ v\end{array}\right)=X\theta(\|X_{c}\|)$ and $\left(\begin{array}{cc}
u'\\ v'\end{array}\right)=X'\theta(\|X_{c}'\|)$.
\end{proof}

We write the linearisation of \eqref{eqme} around the homoclinic orbit $h=\left(\begin{array}{c}
\a\\ \b\end{array}\right)  $. Set  $X=Z+h$ and define  
\begin{equation}\label{def.N}
\mathcal{N}(Z):=F(Z+h)-F(h)-\text{d}F[h].Z,
\end{equation}
with 
\begin{equation*}
\text{d}F[h]=\left(\begin{array}{cc}
0& 0\\-(2p+1)\a^{2p}(t)& 0\end{array}\right).
\end{equation*}
As a result,  $Z$ satisfies the equation 
\begin{equation}\label{eqli}
\partial_{t}Z=\L^{t} (Z)+\mathcal{N}(Z).
\end{equation}

Similarly as in Lemma \ref{4.0} we have
\begin{lemm}\ph\label{4.1}
Consider the function $\mathcal{N}$ defined in \eqref{def.N} and let $Z, Z'\in\X$  with
 $|Z_{h}|,|Z'_{h}|\leq \delta$ and with  $\|Z+h\|,\|Z'+h\|\leq C_{1}$
\begin{equation}\label{tronca0}
\|\underline{\mathcal{N}}(Z)\|\leq  C\delta^{2}  ,
\end{equation}
\begin{equation*} 
\|\underline{\mathcal{N}}(Z)\|\leq  C  \|Z\|^2 ,
\end{equation*}
\begin{equation}\label{tronca}
\|\underline{\mathcal{N}}(Z)-\underline{\mathcal{N}}(Z')\|\leq  C\delta \|Z-Z'\|.
\end{equation}

\end{lemm}

\subsection{The local center-stable manifold}\label{4.2}

The equation \eqref{eqli} is equivalent to the system
\begin{equation}\label{syst.Z}  
\left\{
\begin{aligned}
& \partial^{}_t Z_{h}=\L^{t}_{h}Z_{h}+P\mathcal{N}(Z),\\
&  \partial^{}_t Z_{c}=\L^{t}_{c}Z_{c}+Q\mathcal{N}(Z).
\end{aligned}
\right.
\end{equation} 
We will be interested in solutions to \eqref{syst.Z} for which the center part is small, that's why we introduce the following truncated system
\begin{equation} \label{approx}
\left\{
\begin{aligned}
& \partial^{}_t Z_{h}=\L^{t}_{h}Z_{h}+P\un(Z),\\
&  \partial^{}_t Z_{c}=\L^{t}_{c}Z_{c}+Q\un(Z).
\end{aligned}
\right.
\end{equation} 
We now look for  solutions to \eqref{approx} with small hyperbolic part. 
Fix $0<r<m$ and denote by 
\begin{equation*}
\mathcal{E}^{+}_{r}=\big\{Z\in \mathcal{C}\big(\R_{+},\X\big),\; \|Z\|_{r}:=\sup_{t\geq 0} \e^{-rt}\|Z(t)\|<\infty\big\}.
\end{equation*}
We introduce a new small parameter $\eps$ which will quantify the size of  $Z$, namely $\| Z\|\leq C\eps^2$. We will choose $\eps=\delta^{2/3}$ in such way the truncation $\theta$ can be removed.   
\begin{prop}\ph\label{prop.Z}
Let $V_{c}\in \X_c$ with $\|V_{c}\|\leq \eps^{2}$ and  let $V_s\in\R$ with $|V_{s}|\leq \eps^{2}$. Then there exists a unique solution $Z\equiv Z_{V_c,V_s}$ to \eqref{approx} in 
\begin{equation*}
\mathcal{B}^{+}_{\delta}=\big\{Z\in \mathcal{E}^{+}_{r},\;  \sup_{t\geq 0}|Z_{h}(t)|\leq \delta\big\},
\end{equation*} 
and so that 
\begin{equation}\label{CI}
\<Z_{h}(0),\s^{\star}(0)\>=V_{s},\quad Z_c(0)=V_{c}.
\end{equation}
Moreover, this solution satisfies
\begin{equation}\label{bond}
 \sup_{t\geq 0}|Z_{h}(t)|\leq C\eps^{2},
\end{equation}
and
\begin{equation}\label{petit}
\|Z_{c}\|\leq C\eps^{2},\quad \text{for all}\quad 0\leq t\leq \eps^{-1}.
\end{equation}
\end{prop}
This result means, that for any fixed $\|V_{c}\|\leq \eps^{2}$, $|V_{s}|\leq \eps^{2}$, there exists a unique choice of $V_{u}=\<Z_{h}(0),\rho^{\star}(0)\> \in \R$ so that the corresponding solution has a small hyperbolic part for all times and a central part small for relatively long time (comparing with the time necessary in order that the homoclinic orbit $h$ reaches a neighbourhood of radius $\eps^2$ of the origin). Let us call {\it local center-stable manifold} the set of all the solutions verifying this property:

\begin{defi}
For $\|V_{c}\|\leq \eps^{2}$ and $|V_{s}|\leq \eps^{2}$, denote by $Z_{V_c,V_s}$ the function given by Proposition~\ref{prop.Z} and 
$$X_{V_c,V_s}:=Z_{V_c,V_s}+h.$$
We define $W^{cs}$ by
$$W^{cs}=\bigcup_{\|V_{c}\|\leq\eps^2,|V_s|\leq\eps^2}\big\{ X_{V_{c},V_s}(0)\big\}.$$
Theorem \ref{ThWcs} below will ensure that for $\eps$ sufficiently small, this set is actually the standard center-stable manifold of the equilibrium $0$ for the initial equation \eqref{eqme}. 
\end{defi}

\begin{proof}
We define the map $\mathcal{F} : \mathcal{E}^{+}_{r} \longrightarrow \mathcal{E}^{+}_{r}$
\begin{eqnarray*}
\mathcal{F}Z(t)&=& V_{s}\s(t)+K(t,0)V_{c}+\int_{0}^{t}\<\,P\un{(Z)}(\tau),\s^{\star}(\tau)\,\>\text{d}\tau\,\s(t) \\
&&-\int_{t}^{+\infty}\<\,P\un{(Z)}(\tau),\rho^{\star}(\tau)\,\>\text{d}\tau\,\rho(t)+\int_{0}^{t}K(t,\tau)Q\un{(Z)}(\tau)\text{d}\tau,
\end{eqnarray*}
and show that for $\delta>0$ small enough, $\mathcal{F}$ is a contraction in $\mathcal{B}^{+}_{\delta}$.\\
$\bullet$ To begin with, we show that $\mathcal{F} : \mathcal{B}^{+}_{\delta} \longrightarrow \mathcal{B}^{+}_{\delta}$.  Since $P=\int_{M}$, we have
\begin{equation*}
|\<\,P\un{(Z)},\s^{\star}\,\>|\leq |P\un{(Z)}||\s^{\star}|\leq\|\un{(Z)}\|_{L^{2}(M)}\e^{m\tau}\leq  \|\un{(Z)}\|\e^{m\tau}
\end{equation*}
Let $|Z_{h}|\leq \delta$, then by \eqref{tronca0}, for all $t\geq0$
\begin{eqnarray}\label{dd}
|\big(\mathcal{F}(Z)\big)_{h}(t)|&\leq & C |V_{s}|\e^{-mt} +C\int_{0}^{t}\|\underline{\mathcal{N}}(Z)\|\e^{m\tau}\text{d}\tau \e^{-mt}+C\int_{t}^{+\infty}\|\underline{\mathcal{N}}(Z)\|\e^{-m\tau}\text{d}\tau \e^{mt}\nonumber\\
&\leq &C(\eps^{2}+\delta^{2})\leq C \eps^{2}.
\end{eqnarray}
On the other hand by Proposition \ref{Prop.Lin} and \eqref{tronca0}
\begin{equation}\label{Vc}
\|\big(\mathcal{F}(Z)\big)_{c}(t)\|\leq  C \|V_{c}\|+C\int_{0}^{t}\|\underline{\mathcal{N}}(Z)\|\text{d}\tau\leq C(\|V_{c}\|+\delta^{2}t),
\end{equation}
which in turn implies that $\dis \|\big(\mathcal{F}(Z)\big)_{c}\|_{r}<\infty$.\\
$\bullet$ We now show that $\mathcal{F}$ is a contraction in $\mathcal{B}^{+}_{\delta}$. By Lemma \ref{4.1}
\begin{multline*}
|\big(\mathcal{F}(Z_{1})\big)_{h}(t)-\big(\mathcal{F}(Z_{2})\big)_{h}(t)|\leq \\
\begin{aligned}
&\leq  C\int_{0}^{t}\|\underline{\mathcal{N}}(Z_{1})-\underline{\mathcal{N}}(Z_{2})\|\e^{m\tau}\text{d}\tau \e^{-mt}+C\int_{t}^{+\infty}\|\underline{\mathcal{N}}(Z_{1})-\underline{\mathcal{N}}(Z_{2})\|\e^{-m\tau}\text{d}\tau \e^{mt}\\
&\leq C\delta \e^{-mt}\|Z_{1}-Z_{2}\|_{r}\int_{0}^{t}\e^{(m+r)\tau}\text{d}\tau+C\delta \e^{mt}\|Z_{1}-Z_{2}\|_{r}\int_{t}^{+\infty}\e^{(-m+r)\tau}\text{d}\tau,
\end{aligned}
\end{multline*}
and therefore
\begin{equation}\label{C1}
\|\big(\mathcal{F}(Z_{1})\big)_{h}(t)-\big(\mathcal{F}(Z_{2})\big)_{h}(t)\|_{r}\leq C\delta \|Z_{1}-Z_{2}\|_{r}.
\end{equation}
Similarly
\begin{eqnarray*}
\|\big(\mathcal{F}(Z_{1})\big)_{c}(t)-\big(\mathcal{F}(Z_{2})\big)_{c}(t)\|&\leq & C\int_{0}^{t}\|\underline{\mathcal{N}}(Z_{1})-\underline{\mathcal{N}}(Z_{2})\|\text{d}\tau\\
&\leq &C\delta \|Z_{1}-Z_{2}\|_{r}\int_{0}^{t} \e^{r\tau}\text{d}\tau,
\end{eqnarray*}
and then 
\begin{equation}\label{C2}
\|\big(\mathcal{F}(Z_{1})\big)_{c}(t)-\big(\mathcal{F}(Z_{2})\big)_{c}(t)\|_{r}\leq C\delta \|Z_{1}-Z_{2}\|_{r}.
\end{equation}
Thus,   \eqref{C1} and \eqref{C2} show that $\mathcal{F}$ is a contraction whenever $0<\delta<1$ is small enough, and we can deduce that there exists a unique fixed point $Z\in \mathcal{B}^{+}_{\delta}$. By definition of $\mathcal{F}$, it is clear that this solution satisfies \eqref{CI}.\\
$\bullet$ Finally, in view of the choice $\eps=\delta^{2/3}$, the bound \eqref{bond} comes from \eqref{dd}  and
 \eqref{petit} is a direct consequence of \eqref{Vc}.
\end{proof}

 \subsection{The local center manifold}\label{4.3}
In this subsection we want to construct solutions that remain close to the origin for all time. Actually this will be achieved by constructing a local center manifold.
Again we consider the approximation  of the initial problem \eqref{eqme}, in which we truncate the non-linearity
\begin{equation} \label{approX}
\left\{
\begin{aligned}
& \partial^{}_t X_{h}=\Lambda_{h}X_{h}+P\underline{F}(X),\\
&  \partial^{}_t X_{c}=\Lambda_{c}X_{c}+Q\underline{F}(X).
\end{aligned}
\right.
\end{equation} 
The next result shows that there exist solutions to \eqref{approX} with small  hyperbolic components.  For $0<r<m$ we define

\begin{equation*}
\mathcal{E}_{r}=\big\{X\in \mathcal{C}\big(\R_{+},\X\big),\; \|X\|_{r}:=\sup_{t\in \R}  \e^{-r|t|}\|X(t)\|<\infty\big\}.
\end{equation*}

\begin{prop}\ph\label{prop.X}
Let $V_{c}\in \D_{c}$. Then there exists a unique solution to \eqref{approX} in 
\begin{equation*}
\mathcal{B}_{\delta}=\big\{X\in \mathcal{E}_{r},\;  \sup_{t\in \R}|X_{h}(t)|\leq \delta\big\},
\end{equation*} 
and so that 
\begin{equation}\label{CI*}
 QX(0)=V_{c}.
\end{equation}
Moreover, this solution satisfies
\begin{equation}\label{bond*}
 \sup_{t\in \R}|X_{h}(t)|\leq C\delta^{3}.
\end{equation}
\end{prop}
The set of these solutions is the center manifold of $0$ for equation \eqref{approX} :
\begin{defi}\ph\label{def46}
For  $V_{c}\in \D_{c}$, denote by $X_{V_{c}}$ the function given by Proposition \ref{prop.X}. We define $W^{c}$, a global center  manifold for  \eqref{approX}  (and a  local center  manifold for \eqref{kg}), by
\begin{equation*}
W^{c}=\bigcup_{V_{c}\in\D_c, t\in\R}\big\{ X_{V_{c}}(t)\big\}.
\end{equation*} 
\end{defi}

\begin{proof}
We proceed similarly as in the proof of Proposition \ref{prop.Z}. Denote by $ \s_0= \left(\begin{array}{cc}
1\\ -m\end{array}\right) $ and  $ \rho_0= \left(\begin{array}{cc}
1\\ m\end{array}\right)$. The dual basis of $\{\s_0, \rho_0\}$ reads $\{\s_0^{\star}, \rho_0^{\star}\}$ with $\s_0^{\star}=-\frac{1}{2m}\left(\begin{array}{cc}
-m\\ 1\end{array}\right)$ and $\rho_0^{\star}=\frac{1}{2m}\left(\begin{array}{cc}
m\\ 1\end{array}\right)$. It is clear that  $\dis (\s_0\e^{-mt},\rho_0\e^{mt})$ form a basis of the solutions of $\partial_{t}X_{h}=\Lambda_{h}X_{h}$. Denote by $K$ the propagator of the equation $\partial_{t}X_{c}=\Lambda_{c}X_{c}$. Then for  $X\in   \mathcal{E}_{r}$ we define 
\begin{align}\begin{split}\label{yes}
\mathcal{G}X(t)=& K(t)V_{c}+\int_{-\infty}^{t}\<\,P\underline{F}(X)(\tau),\s_0^{\star}\e^{m\tau}\,\>\text{d}\tau\,\s_0 \e^{-m t}\\
&-\int_{t}^{+\infty}\<\,P\underline{F}(X)(\tau),\rho_0^{\star}\e^{-m\tau}\,\>\text{d}\tau\,\rho_0\e^{m t}+\int_{0}^{t}K(t-\tau)Q\underline{F}(X)(\tau)\text{d}\tau.
\end{split}\end{align}
It is straightforward to check that any fixed point of $\mathcal{G}$ in  $\mathcal{B}_{\delta}$ satisfies \eqref{approX} and $QX(0)=V_{c}$.\\
$\bullet$ Firstly, we show that $\mathcal{G} : \mathcal{B}_{\delta} \longrightarrow \mathcal{B}_{\delta}$. Let $X\in \mathcal{B}_{\delta}$, then by \eqref{Ftronca0}
\begin{eqnarray}\label{br}
|\big(\mathcal{G}(X)\big)_{h}(t)|&\leq & C\int_{-\infty}^{t}\|\underline{{F}}(X)\|\e^{m\tau}\text{d}\tau \e^{-mt}+C\int_{t}^{+\infty}\|\underline{{F}}(X)\|\e^{-m\tau}\text{d}\tau \e^{mt}\nonumber\\
&\leq &C\delta^{3}.
\end{eqnarray}
Next, by  \eqref{Ftronca0} again 
\begin{equation*}
\|\big(\mathcal{G}(X)\big)_{c}(t)\|\leq  C \|V_{c}\|+C\int_{0}^{t}\|\underline{{F}}(X)\|\text{d}\tau\leq C+C\delta^{3}t,
\end{equation*}
which in turn implies that $\dis \|\big(\mathcal{G}(X)\big)_{c}\|_{r}<\infty$. Hence, with \eqref{br} we get  $\mathcal{G}(X)\in \mathcal{B}_{\delta}$.\\
$\bullet$ In a second time, by \eqref{Ftronca} we can write 
\begin{multline*}
|\big(\mathcal{G}(X_{1})\big)_{h}(t)-\big(\mathcal{G}(X_{2})\big)_{h}(t)|\leq \\
\begin{aligned}
&\leq  C\int_{-\infty}^{t}\|\underline{{F}}(X_{1})-\underline{{F}}(X_{2})\|\e^{m\tau}\text{d}\tau \e^{-mt}+C\int_{t}^{+\infty}\|\underline{{F}}(X_{1})-\underline{{F}}(X_{2})\|\e^{-m\tau}\text{d}\tau \e^{mt}\\
&\leq C\delta^{2} \e^{-mt}\|X_{1}-X_{2}\|_{r}\int_{-\infty}^{t}\e^{m\tau+r|\tau|}\text{d}\tau+C\delta^{2} \e^{mt}\|X_{1}-X_{2}\|_{r}\int_{t}^{+\infty}\e^{-m\tau+r|\tau|}\text{d}\tau,
\end{aligned}
\end{multline*}
and therefore
\begin{equation}\label{C1*}
\|\big(\mathcal{G}(X_{1})\big)_{h}(t)-\big(\mathcal{G}(X_{2})\big)_{h}(t)\|_{r}\leq C\delta^{2} \|X_{1}-X_{2}\|_{r}.
\end{equation}
Similarly
\begin{eqnarray*}
\|\big(\mathcal{G}(X_{1})\big)_{c}(t)-\big(\mathcal{G}(X_{2})\big)_{c}(t)\|&\leq & C\int_{0}^{t}\|\underline{{F}}(X_{1})-\underline{{F}}(X_{2})\|\text{d}\tau\\
&\leq &C\delta^{2} \|X_{1}-X_{2}\|_{r}\int_{0}^{t} \e^{r|\tau|}\text{d}\tau,
\end{eqnarray*}
and then 
\begin{equation}\label{C2*}
\|\big(\mathcal{G}(X_{1})\big)_{c}(t)-\big(\mathcal{G}(X_{2})\big)_{c}(t)\|_{r}\leq C\delta^{2} \|X_{1}-X_{2}\|_{r}.
\end{equation}
Thus,   \eqref{C1*} and \eqref{C2*} show that $\mathcal{G}$ is a contraction in $\mathcal{B}_{\delta}$ whenever $0<\delta<1$ is small enough. The bound \eqref{bond*} is given by \eqref{br}.
\end{proof}

We are now able to give a parametrisation of the local center manifold given in Definition  \ref{def46}. Let $\Psi : \mathcal{D}_{c}\longrightarrow \X_{h}$ the map defined by
$$\Psi(V_c)= (X_{V_{c}})_h(0),$$
then 
\begin{equation}\label{def.Wc}
W^{c} =\big\{(\Psi(V_{c}),V_{c}),\; V_{c}\in\D_c \big\},
\end{equation} 
and, in view of \eqref{yes}, $\Psi$ is quadratic at the origin, $\Psi(V_c)=\mathcal{O}(\|V_c\|^2)$. This is a particular case of a result  of Mielke \cite{Miel2}.

As a consequence we can  prove that the origin is Lyapunov stable within the center manifold:
\begin{lemm}\ph\label{Lemma.Lyapu}
Let $X$ be a solution of \eqref{approX} which lies in $W^{c}$. Assume that for some time $t^{\star}>0$, $\|X_{c}(t^{\star})\|\leq C\eps^{2}$. Then for all time $t\in\R$
\begin{equation}\label{lyapu}
\|X_{c}(t)\|\leq C\eps^{2}.
\end{equation}
\end{lemm}

\begin{proof}
It is here convenient to work in the coordinates $(a_{n},b_{n})_{n\geq 1}$. Recall that from \eqref{H0} the Hamiltonian of \eqref{eqme} reads 
\begin{equation*}
H^0=\frac12(b^{2}_{0}-m^{2}a^{2}_{0})+J+\frac1{2(p+1)}\int_{M}\big(a_{0}+U\big)^{2p+2}\text{d}x,
\end{equation*} 
and observe that $J/C\leq \|X_{c}\|^{2}\leq CJ$. We first show that $|H^{0}|\leq C\eps^{4}$. By assumption, $J(t^{\star})\leq C\eps^{4}$, while by \eqref{bond*} we obtain \begin{equation}\label{bn}
|a_{0}(t)|,|b_{0}(t)|\leq C\delta^{6}=C\eps^{9}\leq C\eps^{2}\quad \text{for all}\quad t\in \R.
\end{equation} 
Next, using Lemma \ref{lemsobo} we obtain, for $t=t^{\star}$
\begin{equation*}
\frac1{2(p+1)}\int_{M}\big(a_{0}+U\big)^{2p+2}\text{d}x\leq C \eps^{4},
\end{equation*}
which proves the claim. Next, with \eqref{bn} we can write for all $t\in \R$
\begin{equation*}
J(t)\leq H^{0}-\frac12(b^{2}_{0}-m^{2}a^{2}_{0})\leq C\eps^{4},
\end{equation*}
which implies \eqref{lyapu}.
\end{proof}

\subsection{The local center-stable manifold is global}\label{4.4}

Let $\|V_{c}\|\leq \eps^{2}$, $|V_{s}|\leq \eps^{2}$ and consider the solution $Z=Z_{V_c,V_s}$ to \eqref{approx} given by Proposition \ref{prop.Z}. Set 
\begin{equation}\label{defteps}
t_{\eps}:=\frac4{m}\ln \frac1{\eps}.
\end{equation}
By \eqref{ab}, with this choice we have $|\h(t)|\leq C\eps^{2}$ for all $t\geq t_{\eps}/2$. Next, we define  $\Theta \in \mathcal{C}^{\infty}(\R,\R)$ so that 
\begin{equation*}  
\Theta(t)=\left\{
\begin{aligned}
& 0,\quad  t\leq t_{\eps}/2\\
& 1,\quad  t\geq t_{\eps},
\end{aligned}
\right.
\end{equation*} 
and with $|\Theta'(t)|\leq C/ t_{\eps}$ for all $t\in \R$. We set 
\begin{equation}\label{defY}
Y(t)=\Theta(t)\big(Z(t)+h(t)\big).
\end{equation}

\begin{prop}\ph\label{Prop.Y}
Consider the function $Y$ defined in \eqref{defY}. Then there exists $\wt{X}$ solution of~\eqref{approX} within $ W^{c}$ so that for all $t\in \R$
\begin{equation}\label{propert}
\|(Y-\wt{X})(t)\|\leq C\eps^{-2}\e^{-rt}\quad\mbox{and}\quad |(Y-\wt{X})_{h}(t)|\leq C\eps^{2}.
\end{equation}
 \end{prop}
 
\begin{proof}
The function $Y$ satisfies the equation 
\begin{equation*}
\partial_{t}Y=\Lambda Y+\underline{F}(Y)+\underline{\mathcal{S}},
\end{equation*}
with 
\begin{equation}\label{defS}
\underline{\mathcal{S}}(t):=\Theta'(t)(Z+h)(t)+\Theta(t)  \underline{F}\big(\,(Z+h)(t)\,\big)- \underline{F}\big(\,\Theta(t)(Z+h)(t)\,\big).
\end{equation}
Consider now a solution $\wt{X}$ of the equation $\dis \partial_{t}\wt{X}=\Lambda \wt{X}+\underline{F}(\wt{X})$ and set $\Delta=Y-\wt{X}$. Then $\Delta$ satisfies
\begin{equation}\label{eq.delta}
\partial_{t}\Delta=\Lambda \Delta+\underline{F}(Y)-\underline{F}(Y-\Delta)+\underline{\mathcal{S}}.
\end{equation}
The proof of the proposition consists in finding a solution of \eqref{eq.delta} so that \eqref{propert} holds and so that $\wt{X}(0):=Y(0)-\Delta(0) \in W^{c}$.

Set 
\begin{equation*}
\mathcal{E}^{-}_{r}=\big\{\Delta\in \mathcal{C}\big(\R,\X\big),\; \|\Delta\|_{r}:=\sup_{t\in \R}  \e^{rt}\|\Delta(t)\|<\infty\big\}.
\end{equation*}
and
\begin{equation*}
\mathcal{B}^{-}_{\delta}=\big\{\Delta\in \mathcal{E}^{-}_{r},\;  \sup_{t\in \R}|\Delta_{h}(t)|\leq \delta\big\},
\end{equation*} 
where again we choose $\delta^2=\eps^3$.\\
We  claim that we can define a contraction $\mathcal{J} :  \mathcal{B}^{-}_{\delta} \longrightarrow  \mathcal{B}^{-}_{\delta}$ by 
\begin{eqnarray*}
\mathcal{J}\Delta(t)&=& \int_{-\infty}^{t}\<\,P\big(\,\underline{F}(Y)-\underline{F}(Y-\Delta)+\underline{\mathcal{S}}\,\big)(\tau),\s_0^{\star}\e^{m\tau}\,\>\text{d}\tau\,\s_0 \e^{-m t}\\
&&-\int_{t}^{+\infty}\<\,P\big(\,\underline{F}(Y)-\underline{F}(Y-\Delta)+\underline{\mathcal{S}}\,\big)(\tau),\rho_0^{\star}\e^{-m\tau}\,\>\text{d}\tau\,\rho_0 \e^{m t}\\
&&-\int_{t}^{+\infty}K(t-\tau)Q\big(\,\underline{F}(Y)-\underline{F}(Y-\Delta)+\underline{\mathcal{S}}\,\big)(\tau)\text{d}\tau.
\end{eqnarray*}
It is clear that the fixed point will satisfy \eqref{eq.delta}. In a second time we will check that   \eqref{propert} holds.\\
$\bullet$ We begin by   estimating $\underline{\mathcal{S}}$. First observe that by Proposition \ref{prop.Z} 
\begin{equation}\label{bU}
\|Z(t)\|\leq C\eps^{2} \quad \text{for all}\quad  0\leq t\leq \eps^{-1}.
\end{equation}
Then, for $\eps>0$ small enough we have  $t_{\eps}<\eps^{-1}$ and   we can write for all $t\in [t_{\eps}/2,t_{\eps}]$
\begin{equation}\label{bZ}
\|Z(t)\|\leq \sup_{t\in[t_{\eps}/2,t_{\eps}]}\|Z(t)\|\e^{-r(t-t_{\eps})}\leq C\eps^{2}\e^{-r(t-t_{\eps})}\leq C\eps^{-2}\e^{-rt},
\end{equation}
since $\dis \e^{rt_{\eps}}\leq  \e^{mt_{\eps}}= \eps^{-4}$. Denote by $\chi$ the indicator of the interval $\dis [t_{\eps}/2, t_{\eps}]$. 
As $Z+h$ is small on $[t_{\eps}/2,t_{\eps}]$ we deduce   from \eqref{defS} that
\begin{equation}\label{ss}
\|\underline{\mathcal{S}}(t)\|\leq \frac{C}{t_{\eps}}\big(\|Z(t)\|+\|h(t)\|\big)\chi(t)+C\big(\|Z(t)\|+\|h(t)\|\big)\chi(t).
\end{equation}
As a result, from \eqref{bZ} and \eqref{ab} we infer
\begin{equation}\label{bD}
\|\underline{\mathcal{S}}(t)\| \leq   C(\eps^{-2}\e^{-rt}+\e^{-mt})\chi(t)\leq C\eps^{-2}\e^{-rt}\chi(t).
\end{equation}
The inequality \eqref{ss} together with  \eqref{bU} also gives 
\begin{equation}\label{bornU}
\|\underline{\mathcal{S}}(t)\|\leq C\frac{\eps^{2}}{|\ln \eps\,|}\chi(t).
\end{equation}
$\bullet$ We have   $|Y_{h}|\leq \delta$. For   $\Delta \in \mathcal{B}^{-}_{\delta}$ we can apply \eqref{tronca} to deduce
\begin{equation*}
\|\big(F(Y)-F(Y-\Delta)\big)(\tau)\|\leq C\delta^{2}\|\Delta(\tau)\|\leq C\delta^{2}\|\Delta\|_{r}\e^{-r\tau}.
\end{equation*}
Thanks to this latter inequality and \eqref{bD}, we obtain
\begin{eqnarray*}
\|\mathcal{J}\Delta(t)\|\e^{rt}&\leq&C \int_{-\infty}^{t} \big(\delta^{2}\|\Delta\|_{r}+\eps^{-2}\chi(\tau) \big)\e^{(m-r)\tau}\text{d}\tau\, \e^{(r-m )t}\nonumber\\
&&+C\int_{t}^{+\infty}\big(\delta^{2}\|\Delta\|_{r}+\eps^{-2}\chi(\tau) \big)\e^{-(m+r)\tau}\text{d}\tau\,  \e^{(m+r) t}\nonumber\\
&&+C\int_{t}^{+\infty}\big(\delta^{2}\|\Delta\|_{r}+\eps^{-2}\chi(\tau) \big) \e^{-r \tau}\text{d}\tau\,     \e^{r t}\nonumber\\
&\leq&C(\delta^{2}\|\Delta\|_{r}+\eps^{-2}),
\end{eqnarray*}
and therefore 
\begin{equation}\label{4.44}
\|\mathcal{J}\Delta\|_{r}\leq C(\delta^{2}\|\Delta\|_{r}+\eps^{-2}).
\end{equation}
Next, by \eqref{bornU}, for all $t\in \R$
\begin{eqnarray}\label{4.46}
|(\mathcal{J}\Delta)_{h}(t)|&\leq&C \int_{-\infty}^{t} \big(\delta^{3}+\frac{\eps^{2}}{|\ln \eps\,|}\chi(\tau) \big)\e^{m\tau}\text{d}\tau\, \e^{-m t}\nonumber\\
&&+C\int_{t}^{+\infty}\big(\delta^{3}+\frac{\eps^{2}}{|\ln \eps\,|}\chi(\tau) \big)\e^{-m\tau}\text{d}\tau\,  \e^{mt}\nonumber\\
&\leq&C(\delta^{3}+\eps^{2}).
\end{eqnarray}
The bounds \eqref{4.44} and \eqref{4.46} show that $\mathcal{J} : \mathcal{B}^{-}_{\delta} \longrightarrow \mathcal{B}^{-}_{\delta}$.\\[5pt]
$\bullet$ Let $\Delta_{1},\Delta_{2}\in \mathcal{B}^{-}_{\delta} $. By \eqref{tronca}, we have
\begin{equation}\label{diff}
\|\big(F(Y-\Delta_{2})-F(Y-\Delta_{1})\big)(\tau)\|\leq C\delta^{2}\|\Delta_{1}-\Delta_{2}\|_{r}\e^{-r\tau}.
\end{equation}
Clearly,
\begin{eqnarray*}
\mathcal{J}\Delta_{1}(t)-\mathcal{J}\Delta_{2}(t)&=& \int_{-\infty}^{t}\<\,P\big(\,\underline{F}(Y-\Delta_{2})-\underline{F}(Y-\Delta_{1}),\s^{\star}\e^{m\tau}\,\>\text{d}\tau\,\s \e^{-m t}\\
&&-\int_{t}^{+\infty}\<\,P\big(\,\underline{F}(Y-\Delta_{2})-\underline{F}(Y-\Delta_{1})\,\big)(\tau),\rho^{\star}\e^{-m\tau}\,\>\text{d}\tau\,\rho \e^{m t}\\
&&-\int_{t}^{+\infty}K(t-\tau)Q\big(\,\underline{F}(Y-\Delta_{2})-\underline{F}(Y-\Delta_{1})\,\big)(\tau)\text{d}\tau,
\end{eqnarray*}
and by \eqref{diff}, for all $t\in \R$ 
 \begin{eqnarray*}
\|\mathcal{J}\Delta_{1}(t)-\mathcal{J}\Delta_{2}(t)\|\e^{rt}&\leq&C \int_{-\infty}^{t} \big(\delta^{2}\|\Delta_{1}-\Delta_{2}\|_{r} \big)\e^{(m-r)\tau}\text{d}\tau\, \e^{(r-m )t}\nonumber\\
&&+C\int_{t}^{+\infty}\big(\delta^{2}\|\Delta_{1}-\Delta_{2}\|_{r} \big)\e^{-(m+r)\tau}\text{d}\tau\,  \e^{(m+r) t}\nonumber\\
&&+C\int_{t}^{+\infty}\big(\delta^{2}\|\Delta_{1}-\Delta_{2}\|_{r} \big) \e^{-r \tau}\text{d}\tau\,     \e^{r t}\nonumber\\
&\leq&C\delta^{2}\|\Delta_{1}-\Delta_{2}\|_{r}. 
\end{eqnarray*}
 As a consequence, for $\delta>0$ small enough, the map $\mathcal{J} : \mathcal{B}^{-}_{\delta} \longrightarrow \mathcal{B}^{-}_{\delta}$ is a contraction, thus there exists a unique fixed point $\Delta \in \mathcal{B}^{-}_{\delta}$. 
By \eqref{4.46} and choosing again $\delta^2=\eps^3$, we get $|\Delta_{h}(t)|\leq \eps^{2}$. Furthermore, as $\Delta=\mathcal{J}(\Delta)$, \eqref{4.44} leads to  $\|\Delta\|_{r}\leq C\eps^{-2}$, which in turn implies  $\|\Delta(t)\|\leq C\eps^{-2}\e^{-rt}$ for all $t\in \R$.\\[5pt]
$\bullet$ We now define $\wt{X}=Y-\Delta$ and it remains to show that $\wt{X}\in W^{c}$. By definition, it is sufficient to prove that $|\wt{X}_{h}(t)|\leq \delta$ for all $t$. Write  $Y=\Theta(Z+h)$. By \eqref{bond} and the choice of $t_{\eps}$ we get that $|Y_{h}(t)|\leq C\eps^{2}$  for all $t\in \R$ and then 
\begin{equation*}
|\wt{X}_{h}(t)|\leq |Y_{h}(t)|+ |\Delta_{h}(t)|\leq C\eps^{2}\leq \delta,\quad \text{for all}\quad t\in \R,
\end{equation*}
which implies the result, for $\eps>0$ small enough.
\end{proof} 

\begin{prop}\ph
 Let $|V_{s}|, \|V_{c}\|\leq \eps^{2}$ and $\eps$ small enough. The solution $Z$ of \eqref{approx} with $Z_c(0)=V_c$ and $Z_s(0)=V_s$ constructed in Proposition \ref{prop.Z} 
 satisfies $\|Z_{c}(t)\|\leq \delta$ for all $t\geq 0$. In particular $X=Z+h$ is a solution of \eqref{eqme} for $t\geq 0$. 
\end{prop}  
 
\begin{proof} (see Figure \ref{Schemastrategie} on page \pageref{Schemastrategie})
 For such a $Z$, consider $Y$ defined by \eqref{defY} and $\wt{X}$ defined in Proposition \ref{Prop.Y}. For $t\geq t_{\eps}$, we have $Y=Z+h$, thus  by \eqref{propert} 
 \begin{equation*}
 \|\wt{X}_{c}(t)-Z_{c}(t)\|=\|\wt{X}_{c}(t)-Y_{c}(t)\|\leq \|\wt{X}(t)-Y(t)\|\leq C\eps^{-2}\e^{-rt}, \quad \text{for}\quad t\geq t_{\eps}.
 \end{equation*}
 Recall that by  \eqref{defteps},  $\dis t_{\eps}:=\frac4{m}\ln \frac{1}{\eps}$. Thus, for $\eps>0$ small enough $t_{\eps}\leq \eps^{-1}$ and the previous  inequality implies
 \begin{equation}\label{bble}
  \|\wt{X}_{c}(t)-Z_{c}(t)\|\leq  C\eps^{-2}\e^{-r\eps^{-1}}\leq C\eps^{2}, \quad \text{for}\quad t\geq \eps^{-1}.
 \end{equation}
  
 Now, take $t^{\star}=\eps^{-1}$, by \eqref{petit} we have $\|Z_{c}(t^{\star})\|\leq C\eps^{2}$, and then \eqref{bble} implies $\|\wt{X}_{c}(t^{\star})\|\leq C\eps^{2}$. By Lemma \ref{Lemma.Lyapu}, we deduce that $\|\wt{X}_{c}(t)\|\leq C\eps^{2}$ for all $t\geq \eps^{-1}$, and coming back to \eqref{bble}, we infer
\begin{equation*} 
\|Z_{c}(t)\|\leq C\eps^{2}\leq \delta,\quad \text{for all} \quad t\geq \eps^{-1}.
\end{equation*}
This bound together with \eqref{petit} conclude the proof.
 \end{proof}

Gathering the results of the previous propositions, we get the following Theorem, which ensures that~$W^{cs}$ is the standard center-stable manifold of $0$ for equation \eqref{eqme}. 
\begin{theo}\ph\label{ThWcs}
For $\eps$ sufficiently small, for all $\|V_c\|,|V_s|\leq\eps^2$, the functions
$$X_{V_c,V_s}=Z_{V_c,V_s}+h,$$
(where $Z_{V_c,V_s}$ is given by Proposition \ref{prop.Z}) defined for $t\geq 0$ are solutions of \eqref{eqme}. Moreover, $Z_{V_c,V_s}$ belongs to $\mathcal{E}^{+}_{r}$ and satisfy
$$\|Z_c(t)\|\leq C\eps^{2}\leq \delta, \quad \|Z_h(t)\|\leq C\eps^2, \quad \text{for all }t\in\R^+ ;$$
and there exists $\wt{X}\in W^c$ such that
$$\|(X_{V_c,V_s}-\wt{X})(t)\|\leq C\eps^{-2}\e^{-rt}.$$
\end{theo}

This completes the proof of Theorem \ref{thm1}.

\subsection{Construction of reversible heteroclinic solutions: proof of Theorem \ref{coro.thm}}\label{4.5}

Recall that the vector field is reversible, and thus if the initial condition of a solution $X$ satisfies
$$X(0)=SX(0),$$
then the solution is reversible, $i.e.$
$$X(-t)=SX(t), \quad \text{for all }t\in\R.$$

\begin{lemm}\ph
For $\eps$ sufficiently small, for all $\|V_c\|\leq\eps^2$ verifying
\begin{equation*} 
V_c=SV_c,
\end{equation*}
$X_{V_c,V_s}$ defined in Theorem \ref{ThWcs} satisfies
$$X_{V_c,0}(0)=SX_{V_c,0}(0).$$
\end{lemm}
\begin{proof}
On the one hand, from the explicit form \eqref{def.homo} of $h$, we know that $h(0)=Sh(0)$. On the other hand, given that $Z$ is a fixed point of the map $\mathcal{F}$ defined in the proof of Proposition \ref{prop.Z}, we get
$$Z_{V_c,0}(0)=V_c-\int^{+\infty}_{0}\<\,P\un{(Z)}(\tau),\rho^{\star}(\tau)\,\>\text{d}\tau\,\rho(0).$$
And since $\rho(0)=\left(\begin{array}{cc}
1 \\ 0 \end{array}\right)$ (see Lemma \ref{LemLin}), we have  $S\rho(0)=\rho(0)$. So, if $V_c=SV_c$, then $Z_{V_c,0}(0)=SZ_{V_c,0}(0)$.
\end{proof}

As a consequence of the previous lemma, $X_{V_c,0}$ is a reversible solution.
 From Theorem \ref{ThWcs}, we then get that on the one hand, $Z_{V_c,0}=X_{V_c,0}-h$ belongs to $\mathcal{E}_{r}$ and satisfy
$$\|(Z_{V_c,0})_c(t)\|\leq\delta, \quad \|(Z_{V_c,0})_h(t)\|\leq C\eps^2, \quad \text{for all }t\in\R ;$$
and on the other hand, there exists $X\in W^c$ such that
$$\|(X_{V_c,0}-X)(t)\|\leq C\eps^{-2}\e^{-rt} \quad \text{for all}\quad t\geq0,$$
and 
$$ \|(X_{V_c,0}-SX)(t)\|\leq C\eps^{-2}\e^{rt}  \quad \text{for all}\quad t\leq0.$$
This means that $X_{V_c,0}$ is an heteroclinic connection between the two solutions $X$ and $SX$ of the center manifold $W^c$.\\
Then Theorem \ref{coro.thm} follows since the condition $V_c=SV_c$ and $V_{s}=0$ is equivalent to \eqref{sol3}.


\appendix

\section{}

In this appendix we recall a result  concerning the long time behaviour of the solution of the ordinary differential equation
\begin{equation*} 
\ddot x= (-\alpha^2+q(t))x,\quad t\in\R
\end{equation*}
where $\alpha$ is a real constant and $t\mapsto q(t)$ is a continuous function. Of course the solution of such linear equation are globally defined, but we would like to know whether the solution are bounded or not. It turns out that even if $\dis \|q\|_\infty < \alpha^2$ and $q(t)\longrightarrow 0$ when $t\longrightarrow \pm \infty$, the solutions may grow indefinitely\footnote{Consider the example induced by the solution $x(t)= \sin t (1+(t-\sin t\cos t)^2)$ (see O. Perron \cite{Perron}). }. The good condition concerns the integrability of $\dis |q|$:
\begin{lemm}\ph\label{A2}
Assume that $\dis \int_\R |q(t)|\text{d}t<+\infty$ then all the solutions of the Cauchy problem
\begin{equation}\label{eq.appen}
\left\{
\begin{aligned}
&\dot x= y\\
& \dot y=(-\alpha^2+q(t))x,
\end{aligned}
\right.
\end{equation}
 are bounded on $\R$. More precisely, there exists $C>0$ which only depends on $q$ so that
 $$
\big(\a^{2}|x(t)|^2+|y(t)|^2\big)\leq C\big(\a^{2}|x(0)|^2+|y(0)|^2\big)\quad \text{for all } t\in\R.
$$
\end{lemm}
The proof is classical (see for instance \cite[page 212]{HS}) but  does not precise the bound on the solution in term of the initial datum. The following argument is more explicit.

\proof
Denote by $S$ the flow of \eqref{eq.appen} for $q\equiv 0$, and  introduce the norm $\|x,y\|=(\a^{2}x^{2}+y^{2})^{1/2}$ (observe that $S$ preserves this norm). For $T\geq 0$,  the Duhamel formula reads
\begin{equation}\label{duha.T}
\left(\begin{array}{c}
x(t)\\y(t)\end{array}\right) =S(t-T)\left(\begin{array}{c}
x(T)\\y(T)\end{array}\right)+\int_T^t S(t-T-s)\left(\begin{array}{c}
0\\q(s)x(s)\end{array}\right)\text{d}s.
\end{equation}
There exists $T_{1}$ so that  $\dis \int_{0}^{T_{1}} |q(t)|\text{d}t\leq 1/2$. Therefore from \eqref{duha.T}  with $T=0$ we deduce that for all $t\in [0,T_{1}]$
$$
\|x(t),y(t)\|\leq \|x(0),y(0)\|+ \frac12\max_{s\in [0,T_{1}]}\|x(s),y(s)\|, $$
and thus $\dis \max_{s\in [0,T_{1}]}\|x(s),y(s)\|\leq 2\|x(0),y(0)\|.$\\
By induction, we can define a finite number of times $T_{2},\dots, T_{k}$ so that $\dis \int_{T_{j}}^{T_{j+1}} |q(t)|\text{d}t=1/2$ and $\dis \int_{T_{k}}^{+\infty} |q(t)|\text{d}t\leq1/2$. Then we apply  \eqref{duha.T} with $T=T_{j}$  and we show $\dis \max_{s\in [T_{j},T_{j+1}]}\|x(s),y(s)\|\leq 2\|x(T_{j}),y(T_{j})\|$ which in turn implies
$$
\|x(t),y(t)\|\leq 2^{k+1}\|x(0),y(0)\| \quad \text{for all } t\geq 0.
$$
\endproof

\end{document}